\documentclass[]{article}
\usepackage{amsmath}
\usepackage{amssymb}
\usepackage{mathtools}
\usepackage{graphics}
\usepackage{amsfonts}
\usepackage{tikz} 
\usepackage{forloop}
\usepackage{multirow}
\usepackage{float}
\usepackage[a4paper]{geometry}
\usepackage[toc,title,page]{appendix}
\usepackage{amsthm}
\usepackage{parskip}
\usepackage{mathrsfs}
\usepackage{graphicx}
\usepackage{xtab}
\usepackage{tikz}
\usepackage{forloop}
\usepackage{multirow}
\usepackage{caption}
\usepackage{subcaption}
\usepackage{enumerate}
\usepackage{amsthm}
	\usepackage[utf8]{inputenc}
	\usepackage{setspace}
 
\usepackage[
backend=biber,
style=numeric,
sorting=none
]{biblatex}

\usepackage[section]{placeins}
\usepackage{enumitem}

\usepackage[capitalise]{cleveref}

\usetikzlibrary{automata,positioning,arrows}
\tikzset{node distance=1.5cm}
\title{Adaptive Dynamics of Diverging Fitness Optima}
\author{}

\newtheorem{lemma}{Lemma}

\newtheorem{theorem}{Theorem}

\newtheorem{cor}{Corollary}

\theoremstyle{remark}
\newtheorem{remark}{Remark}

\newtheorem{defn}{Definition}

\usepackage [english]{babel}
\usepackage [autostyle, english = american]{csquotes}
\MakeOuterQuote{"}

\author{
  Manh Hong Duong\\
  \texttt{h.duong@bham.ac.uk}
  \and
  Fabian Spill\\
  \texttt{f.spill@bham.ac.uk}
    \and
 Blaine Van Rensburg\\
\texttt{BXV114@student.bham.ac.uk}
}



\begin{document}
\maketitle

\begin{abstract}
    We analyse a non-local parabolic integro-differential equation modelling the evolutionary dynamics of a phenotypically-structured population in a changing environment. Such models arise  in a variety of contexts from climate change to chemotherapy to the ageing body. The main novelty is that there are two locally optimal traits, each of which shifts at a possibly different linear velocity. We determine sufficient conditions to guarantee extinction or persistence of the population in terms of associated eigenvalue problems. When it does not go extinct, we analyse the solution in the long time, small mutation limits. If the optimas have equal shift velocities, the solution concentrates on a point set of "lagged optima" which are strictly behind the true shifting optima. If the shift velocities are different, we determine that the solution in fact concentrates as a Dirac delta function on the positive lagged optimum with maximum lagged fitness, which depends on the true optimum and the rate of shift. Our results imply that for populations undergoing competition in temporally changing environments, both the true optimal fitness and the required rate of adaptation for each of the diverging optimal traits contribute to the eventual dominance of one trait.
\end{abstract}

\section{Introduction}
\subsection{Model and main questions}

The non-local reaction-diffusion equation considered in this paper models an asexual population undergoing natural selection in a changing environment in the presence of multiple traits that give a locally optimal reproduction rate. Each optimum shifts due to the environment, possibly with different velocities. Generally, we study models of the form:
\begin{equation}\label{eqn:GeneralAdapativeDyn}
\begin{cases*}
        \partial_{t}n(x,t)-\sigma\partial_{xx}{n(x,t)}=n(x,t)\left(A(x,t)-\rho_\varepsilon(t)\right), &$(x,t)\in{\mathbb{R}}\times{}[0,\infty),$ \\
        n(x,0)=n_{0}(x), & $x\in\mathbb{R}$.
\end{cases*}  
\end{equation}
where $n(x,t)$ represents the concentration of individuals with trait $x\in\mathbb{R}$ present at time $t$, and $\rho_\varepsilon(t)=\int_{\mathbb{R}}n(y,t)dy$ is the total population. The function $A(x,t)$ is the intrinsic growth rate of individuals with trait $x$ at time $t$, where we interpret the time dependence to be due to the changing environment. At a fixed time, the graph of $A(x,t)$ is called the fitness landscape  \cite{figueroa2018long,lorenzi2020asymptotic,iglesias2021selection,}. The evolutionary concept of fitness is nuanced, but in the context of this work it can be regarded as synonymous with relative growth rate. The growth rate is modified by competition so that  $A(x,t)-\int_{\mathbb{R}}n(y,t)dy$ is the per capita growth rate modified by competition, i.e Lotka-Volterra where the competition is across traits. Movement in trait-space is due to genetic alterations such as mutations. For simplicity, we model this using the diffusion term $\sigma\Delta{n}$ where the diffusive coefficient is given by $\sigma$. This is a common assumption, but other works have also considered more general mutational kernels as in \cite{carrillo2007adaptive} or \cite{diekmann2005dynamics}. Similar models have been studied, under differing assumptions on the per capita growth term $A(x,t)$ in many works \cite{lorenzi2020asymptotic,iglesias2021selection,Pouchol2018,magal2000mutation,roques2020adaptation}.

The case of multiple globally optimal traits is investigated in \cite{lorenzi2020asymptotic} in a static environment. The case of a single optimal trait in a shifting and periodic environment is investigated in \cite{iglesias2021selection} which builds on the work \cite{figueroa2018long} that considers a periodic environment without shift. In the case of \cite{iglesias2021selection} the results pertain to the question of whether a species can adapt fast enough to a changing environment to survive. This question of persistence has also been investigated in \cite{alfaro2017effect}, which considers a spatially shifting environment and mutations in phenotype, and \cite{berestycki2018forced}  in a purely spatial context. The authors of \cite{iglesias2021selection} find conditions on the rate of environmental shift and the maximum fitness to determine whether the population goes extinct or not. They assume that the trait- and time-dependent per capita growth rate $A(x,t)$ is given by a function $a(x-\tilde{c}t,e(t))$ where $e(t)$ is a periodic function with period $T$, and the average fitness $\bar{a}(x)=\frac{1}{T}\int{}a(x,e(t))$ has just one global optimum. By using a growth function with this linearly shifting form, it is implicitly assumed that the changing environment shifts locally optimal phenotypes equally. In other words, a given change in the environment would shift all local optima by the same distance. In reality, the change of the trait and time-dependent per capita growth function may change in much more complicated ways.

In the present work, we take a more general form of the intrinsic growth rate that allows for two locally optimal traits each of which shifts linearly at possibly different rates. This leads to the following natural questions:
\begin{enumerate}
    \item For which conditions on the shifting rate and intrinsic growth rates does the entire population go extinct?
    \item Which trait/s dominate in long time, if the population does not go extinct?
\end{enumerate}
We are ultimately interested in the situation where $A(x,t)$  has multiple time-dependent optima which shift at constant speeds in different directions, representing alternative evolutionary trajectories (i.e representing alternative phenotypes which adapt to the changing environment). To do this we study two distinct cases.

\textbf{Case 1:} For the case of multiple peaks shifting in the same direction at the same speed, the model \eqref{eqn:GeneralAdapativeDyn} reduces to the following PDE:
\begin{equation}\label{eqn:UnboundedShifting}
\begin{cases*}
        \partial_{t}n-\sigma\partial_{xx}n=n\left(a(x-\tilde{c}t)-\rho_\varepsilon(t)\right), &$(x,t)\in{\mathbb{R}}\times{}\mathbb{R}^{+},$ \\
        n(x,0)=n_{0}(x) 
\end{cases*}  
\end{equation}
Here $\tilde{c}>0$ can be interpreted as the rate at which the fitness of phenotypes responds to environment shift. We take the scaling $\sigma=\varepsilon^2$ and $\tilde{c}=\varepsilon{}c$ where $c$ is constant, and analyse the behaviour of the solution $n_\varepsilon(x,t)$ for rare mutations, i.e for $\varepsilon\rightarrow{0}$.
We assume that there are only finitely many maxima of $a(x)$ and we let $M=\{x_{1},x_{2},\dots,x_{n}\}=\text{argmax}_{x\in\mathbb{R}} a(x)$, where $x_{i}$ are distinct points. We also assume there is a fixed $\delta>0$ such that $a(x)<-\delta$ for  $|x|$ sufficiently large. We let $a_{M}:=\max_{x\in\mathbb{R}}a(x)$. 

\textbf{Case 2:} In the case where the optimal traits diverge, we will study the following PDE:
\begin{equation}\label{eqn:DivergingOptima}
\begin{cases*}
        \partial_{t}n-\sigma\partial_{xx}{n}=n\left(a_{1}(x-\tilde{c}_{1}t)+a_{2}(x-\tilde{c}_{2}t)-\delta-\rho_\varepsilon(t)\right), &$(x,t)\in{\mathbb{R}}\times{}\mathbb{R}^{+},$ \\
        n(x,0)=n_{0}(x). &
\end{cases*} 
\end{equation}
Here we assume $\tilde{c}_{1}<0<\tilde{c}_{2}$, and that $\delta$ is a positive constant. We assume each $a_{i}$ is continuous, compactly supported in $[-R_{0},R_{0}]$, and has a unique maximum value $a_{i,M}$ at $x=x_{i}$. We again take the scaling $\sigma=\varepsilon^2$,  $\tilde{c}_i=\varepsilon{}c_i$ and consider the behaviour of the solution $n_\varepsilon(x,t)$ in the small mutation limit.

\subsection{Overview of results and structure of paper}
The main results of the paper are as follows. For Case 1, \cref{thm:EigenvalueConvergence} provides neccesary and sufficient conditions that determine whether the population, in small mutation limit, either goes extinct for large time, or concentrates on a point set. We expect this set of concentration points can be further refined, and \cref{thm:RescaleConvergence} shows that a particular weighted rescaling of the limiting solution will concentrate only on the shallowest peak, i.e the $x_{i}\in\text{argmax}_{x\in\mathbb{R}}a(x)$ which minimises $|a''(x)|$. This suggests a peculiar result: the subpopulation which persists is the one which follows the shallowest moving optimum, even though that subpopulation may be itself concentrated at a trait where $|a''(x)|$ is large. This is not intuitive because it suggests the subpopulation benefits from the trait $x_{i}$ that none of the numbers of the limiting population have, since they are concentrated away from the moving optimum.

In Case 2, \cref{thm:TwoPeaksMainResult} provides necessary and sufficient conditions to determine whether the population goes extinct or persists. When it does persists, we determine explicitly the concentration point in terms of the shifting speeds and intrinsic growth rates $a_i(x)$. The interpretation is more straightforward here: a subpopulation following a particular moving optimal trait, even if not globally optimal, can persist provided the alternative trait is too fast.

The paper is structured as follows. In \cref{sec:StrategyOfProof} we will present our main results in more detail, which answer the questions we have asked regarding persistence and extinction, and outline the strategy for proving them. In \cref{sec:MainResults} we will prove  \cref{thm:EigenvalueConvergence,thm:RescaleConvergence,thm:TwoPeaksMainResult} and in \cref{sec:NumericalResults} we will illustrate our results with numerical simulations (which provide some insight into the transient dynamics which are not captured by our theorems). In \cref{sec:Discussion} we will discuss the biological interpretation of the results and suggest some future directions. For the sake of being self-contained, we collect some results from the literature in \cref{sec:Preliminaryresults}. Moreover, this provides examples of the sort of results one can expect for this problem. We
confine the more technical proofs to \cref{Appendix}.

\subsection{Biological relevance}
Our results  may  be relevant for cancer biology, with regards to the phenomenon of decoy fitness peaks \cite{higa2019decoy}. To summarise it, the hypothesis is that certain mutations (for instance in the NOTCH1 gene) may be non-cancerous but enable the mutants to survive better in an aged microenvironment allowing them to compete with pre-cancerous cells and thus suppress the development of tumours. Recent experimental evidence in support of this hypothesis can be found in \cite{colom2021mutant}. This phenomenon involves competition between two optimal phenotypes each of which is adapting to a time-dependent environment (in this case, due to ageing). Although our focus here is on the relevance of differing rates of adaption, we suggest that this framework offers a useful starting point for modelling situations where asexual populations are in competition in a temporally changing environment. 

\section{Statement of results and strategy of proof}\label{sec:StrategyOfProof}
In this section, we set up the problem more thoroughly in each of the two cases so that we can state our main theorems precisely. We also detail our strategies for proving these results. In all cases, we rely on the notion of generalised super- and sub- solutions as provided in \cite{lam2022introduction}, and the results on principle Floquet bundles from \cite{huska2008exponential}. These results are reviewed in, respectively, \cref{subsec:Prelim0} and \cref{subsec:Prelim}.
\subsection{Case 1: Two peaks shifting in with the same velocity}\label{subsubsec:SameVelocityIntro}

The system  \eqref{eqn:UnboundedShifting}, after taking the appropriate scaling of the mutation and drift terms ($\tilde{c}=c\varepsilon$, and $\sigma=\varepsilon^2$), becomes
\begin{equation}\label{eqn:UnboundedShiftingEps}
\begin{cases*}
        \partial_{t}n_\varepsilon-\varepsilon^{2}\partial_{xx}n_\varepsilon=n_\varepsilon\left(a(x-c\varepsilon{t})-\rho_\varepsilon(t)\right), &$(x,t)\in {\mathbb{R}}\times{}\mathbb{R}^{+},$ \\
        n_\varepsilon(x,0)=n_{0}(x),
\end{cases*}  
\end{equation}
where $\rho_\varepsilon(t)=\int_{}n_\varepsilon(y,t)dy$ is the total size of the population.

We make the following assumptions:
\begin{enumerate}[label=(A\arabic*)]
\item     There exist $R_{0},\delta>0$ such that $a(x)<-\delta$ provided $|x|>R_{0}$
\label{assum:A1}
\item $a(x)\in C^{2}(\mathbb{R})$ and $\Vert{}a\Vert_{L^\infty(\mathbb{R})}\leq{}d_{0}$.\label{assum:A2}
\item $0\leq{}n_{0}\leq{}e^{C_{1}-C_2|x|}$ for some positive constants $C_{1}$ and $C_{2}$, and $n_{0}\in{}C(\mathbb{R}).$ \label{assum:A3}
\item There are finite number of global maxima of $a$.\label{assum:A4}
\end{enumerate}
To state the next, and final, assumption succinctly, we must introduce some notation. We recall $M=\{x_{1},x_{2},\dots,x_{n}\}=\text{argmax}_{x\in\mathbb{R}}a(x)$, and that $a_{M}$ is the maximum of $a(x)$. We assume the optima are ordered $x_{i}<x_{i+1}$ for each $i=1,...,n-1$. Our final assumption is then

\begin{enumerate}[label=(A\arabic*)]
\setcounter{enumi}{4}
    \item For $i=2,...,n$ there are only two solutions to $a(x)=a_{M}-\frac{c^2}{4}$ in the interval $(x_{i-1},x_{i})$, and let $\bar{x}_{i}$ be the greatest of these. There is just one solution $\bar{x}_{1}$ in the interval $(-\infty,x_{1})$.\label{assum:A5}
\end{enumerate}

In preparation for the statement of our first result, we first consider some transformed problems and the associated eigenvalue problems. Firstly, we take $N_\varepsilon(x,t)=n_\varepsilon(x+\varepsilon{c}t,t)$ so that $N_\varepsilon(x,t)$ solves:
\begin{equation}\label{eqn:SameVelocityShifted}
\begin{cases*}
        \partial_{t}N_\varepsilon-\varepsilon^{2}\partial_{xx}N_\varepsilon-\varepsilon{c}\partial_{x}N_\varepsilon=N_\varepsilon\left(a(x)-\int_{\mathbb{R}}N_\varepsilon(y,t)dy\right), &$(x,t)\in {\mathbb{R}}\times{}\mathbb{R}^{+},$ \\
        N_\varepsilon(x,0)=n_{0}(x).
\end{cases*}  
\end{equation}
Next one can linearlise this by working with the equation for $m_\varepsilon(x,t)=N_\varepsilon{}(x,t)e^{\int_{0}^{t}\rho_{\varepsilon}(s)ds}$.
\begin{equation}\label{eqn:Linearised}
\begin{cases*}
       \partial_{t}m_\varepsilon-c\varepsilon\partial_{x}m_\varepsilon-\varepsilon^2\partial_{xx}m_\varepsilon=a(x)m_\varepsilon, & $z\in {\mathbb{R}},$ \\
        m_\varepsilon(x,0)=n_{0}(x).\\
\end{cases*}  
\end{equation}
Finally, one can remove the drift term by applying a Liouville transform $M_\varepsilon(x,t)=m_\varepsilon(x,t)e^\frac{cx}{2\varepsilon}$:
\begin{equation}\label{eqn:WithoutDrift}
\begin{cases*}
           \partial_{t}M_\varepsilon-\varepsilon^2\partial_{xx}M_\varepsilon=M_\varepsilon\left(a(x)-\frac{{c}^2}{4}\right), & $(x,t)\in {\mathbb{R}}\times{}\mathbb{R}^{+},$ \\
        M_\varepsilon(x,0)=n_{0}(x)e^{\frac{{c}x}{2\varepsilon}}.\\
\end{cases*}  
\end{equation}
Associated to \eqref{eqn:Linearised} and \eqref{eqn:WithoutDrift} (respecitvely)
are the following stationary eigenvalue problems
\begin{equation}\label{eqn:MoreGeneralProblemShift}
   \begin{cases*}
        -\varepsilon^2\partial_{xx}{{p}_\varepsilon}-\varepsilon{}c\partial_{x}{p}_\varepsilon-a(x){{p}_\varepsilon}={\lambda}_{\varepsilon}{{p}_\varepsilon}, &$x\in {\mathbb{R}},$ \\
        {p_\varepsilon}>0, & {}
                 \end{cases*}   
\end{equation}
and
\begin{equation}\label{eqn:MoreGeneralProblem}
   \begin{cases*}
        -\varepsilon^2\partial_{xx}{P_\varepsilon}-\varepsilon{}c\partial_{x}P_\varepsilon-\left(a(x)-\frac{c^2}{4}\right){P_\varepsilon}={\lambda}_{c,\varepsilon}{P_\varepsilon}, &$x\in {\mathbb{R}},$ \\
        {P_\varepsilon}>0. & {}
                 \end{cases*}   
\end{equation}
Our main theorem says that the solution of \eqref{eqn:SameVelocityShifted} at long times is approximately given by a multiple of eigenvector which solves \eqref{eqn:MoreGeneralProblemShift}, and does not go extinct if the eigenvalue is negative. The eigenvalue itself is shown to converge to $a_{M}-\frac{c^2}{4}$ so for sufficiently small $\varepsilon$ we can determine its sign.
\begin{theorem}
    
\label{thm:EigenvalueConvergence}
Under the assumptions \labelcref{assum:A1,assum:A2,assum:A3,assum:A4,assum:A5}:

There is a solution ${p}_\varepsilon(x)\in{}L^{1}(\mathbb{R})$ to \eqref{eqn:MoreGeneralProblemShift} which is unique if we specify  $\int_{\mathbb{R}}p_\varepsilon(y)dy=1$. For non-negative constants $a_{i}$ which sum to $1$ we have that \[{p}_\varepsilon\xrightharpoonup[\varepsilon\rightarrow{0}]{}\sum_{i=1}^{n}{a_{i}\delta_{\bar{x}_{i}}},\] 
\[\lambda_\varepsilon\xrightarrow[\varepsilon\rightarrow{0}]{}-\left(a_{M}-\frac{c^2}{4}\right),\] and, if $\lambda_\varepsilon<0$, we also have \[\rho_\varepsilon(t)\xrightarrow[t\rightarrow\infty]{}\int_{\mathbb{R}}a(y)p_\varepsilon(y)dy,\] 
and
\[\left\Vert{\frac{n_\varepsilon(x+\varepsilon{}ct,t)}{\rho_{\varepsilon}(t)}-{p}_{\varepsilon}}\right\Vert_{L_{\infty}(\mathbb{R})}\xrightarrow[t\rightarrow\infty]{}0.\]
 If $\lambda_\varepsilon<0$, then $\Vert{n_\varepsilon(.,t)}\Vert_{L^\infty{(\mathbb{R}})}\xrightarrow[t\rightarrow\infty]{}0$.

\end{theorem}
Here $\delta_{x_{i}}$ is the Dirac delta measure centred at $x_{i}$.
This theorem generalises some of the results from \cite{lorenzi2020asymptotic} and \cite{iglesias2021selection}; specifically it characterises the long-term behaviour of solutions in the novel case of a linearly shifting growth rate which has multiple global optima. Because $p_\varepsilon$ concentrates on the point set $\{\bar{x}_1,...,\bar{x}_n\}$ which we refer to as the set of lagged optima since each is $\bar{x}_i$ is a candidate location for the solution to concentrate.

To prove \cref{thm:EigenvalueConvergence}, we  use the theory developed in \cite{huska2008exponential}, which we review in \cref{subsec:Prelim}, to determine the long term behaviour of solutions to \eqref{eqn:Linearised} in terms of the solution to the associated eigenvalue problem \eqref{eqn:MoreGeneralProblemShift}. The solutions of the eigenvalue problems are analysed using the Hamilton-Jacobi approach.

Although this theorem gives some information about the limiting solution we expect it can be made stronger by further restricting the set of points where the solution concentrates. This will be shown via numerical simulations. We note that in \cite{lorenzi2020asymptotic} the set of limit points is refined further, according to Proposition 3 in \cite{lorenzi2020asymptotic} (restated as \cref{lma:TopologicalPressure} in \cref{sec:Preliminaryresults}).

While unable to fully classify the limit measure we are able to show the convergence of the transformed and rescaled eigenvalue defined as
 \[\hat{p}_\varepsilon=\frac{{p}_\varepsilon{}e^{\frac{cx}{2\varepsilon}}}{\Vert{}{p}_\varepsilon{}e^{\frac{cx}{2\varepsilon}}\Vert_{L_{1}(\mathbb{R})}}.\]
We show that $\Vert{}{p}_\varepsilon{}e^{\frac{cx}{2\varepsilon}}\Vert_{L_{1}(\mathbb{R})}$ is finite in \cref{sec:MainResults}. Then by borrowing a result from semi-classical analysis (as is done in \cite{lorenzi2020asymptotic}) we prove our next theorem.

\begin{theorem}\label{thm:RescaleConvergence}
       Let $S(x)=|a''(x)|$ for $x\in{}M:=\{x_{1},...,x_{n}\}=\text{argmax}_{x\in\mathbb{R}}a(x)$. Let $M_{1}=\text{argmin}_{x_{j}}S(x_{j})=\{x_{i_{1}},x_{i_{2}},...,x_{i_{k}}\}$.

       Then, up to extraction of subsequences,
   \[\hat{p}_\varepsilon\xrightharpoonup[\varepsilon\rightarrow{0}]{}\sum_{j}a_{j}\delta_{x_{i_{j}}},\]
   where $a_{j}>0$ for each $j$ and $\sum_{j=1}^{k}a_{j}=1$.
\end{theorem}

 The main difficulty in proving \cref{thm:RescaleConvergence} lies in the fact that we work on an unbounded domain, whereas Proposition 3 from \cite{lorenzi2020asymptotic} is obtained for problems in a bounded domain (or on a compact Riemannian manifold in \cite{holcman2006singular}). To overcome this, we will use the results in \cite{huska2008exponential} to construct the solution to \eqref{eqn:MoreGeneralProblem} as the limit of solutions to the following Dirichlet problems in bounded domains:

\begin{equation}\label{eqn:DirichletPrince}
   \begin{cases*}
        \partial_{t}P_{R,\varepsilon}-\varepsilon^2\partial_{xx}P_{R,\varepsilon}-\left(a(x)-\frac{c^2}{4}\right)P_{R,\varepsilon}=\lambda_{R,\varepsilon}P_{R,\varepsilon}, &$(x,t)\in{}B_{R}\times{\mathbb{R}^{+}},$ \\ 
        P_{R,\varepsilon}>0, &$(x,t)\in{}B_{R}\times{\mathbb{R}^{+}},$\\
        P_{R,\varepsilon}=0, &$(x,t)\in{}\partial{}B_{R}\times{\mathbb{R}^{+}}$ 
                 \end{cases*}   
\end{equation}

Such a construction is also used in \cite{iglesias2021selection}, and in our proof of \cref{thm:EigenvalueConvergence}. Fortunately, this procedure also enables us to approximate the eigenfunction-eigenvalue pair $(P_\varepsilon,\lambda_\varepsilon)$ in terms of $(P_{R,\varepsilon},\lambda_{R,\varepsilon})$. We can apply the same semi-classical analysis results as used in \cite{lorenzi2020asymptotic}, which require us to work in a bounded domain, to these approximate problems, allowing us to refine the set of concentration points of $P_\varepsilon$ as $\varepsilon\rightarrow{0}$.

The procedure of estimating the problem on the whole real line can similarly be used to estimate $\lambda_\varepsilon$ and obtain its limiting value as $\varepsilon\rightarrow{0}$ which will be used to determine the limiting set. We show later that the convergence of $\lambda_{R,\varepsilon}$ as $\varepsilon\rightarrow{0}$ is uniform in $R$, which allows us to adapt the methods used in \cite{lorenzi2020asymptotic} to obtain the required concentration result.

\subsection{Case 2: Two peaks shifting at different speeds}\label{subsubsec:DiffVelocityIntro}

In this case, we consider \eqref{eqn:GeneralAdapativeDyn} and set $a(x,t)=a_{1}(x-\varepsilon{c}_{1}t)+a_{2}(x-\varepsilon{c}_{2}t)-\delta$. We arrive at the equation
\begin{equation}
\begin{cases*}
        \partial_{t}n_\varepsilon-\varepsilon^2\Delta{n_\varepsilon}=n_\varepsilon\left(a_{1}(x-\varepsilon{c}_{1}t)+a_{2}(x-\varepsilon{c}_{2}t)-\delta-\rho_\varepsilon(t)\right), &$(x,t)\in{\mathbb{R}}\times{}\mathbb{R}^{+},$ \\
        n_\varepsilon(x,0)=n_{0}(x).
\end{cases*}  \label{eqn:TwoPeaksDifferentDirections}
\end{equation}
Here we assume the following:

\begin{enumerate}[label=(B\arabic*)]
    \item The functions $a_{1}$ and $a_{2}$ are in $C^{1}_{c}([-R_{0},R_{0}])$, $\Vert{}a_{1}\Vert_{L^{\infty}(\mathbb{R})}+\Vert{}a_{2}\Vert_{L^{\infty}(\mathbb{R})}-\delta<d_{0}$ for some constant $d_{0}>0$.
  Each $a_{i}$ has a single positive maximum.\label{assum:B1}
    \item The shifting rates satisfy $c_{1}<0<c_{2}$.\label{assum:SignsOfShift}
    \item  $0\leq{}n_{0}\leq{}e^{C_{1}-C_{2}|x|}$ for some positive constants $C_{1}$ and $C_{2}$.\label{assum:B3}
\end{enumerate}
We let $x_{1}=\text{argmax }a_{1}(x)$ and $x_{2}=\text{argmax }a_{2}(x)$ and denote $a_{i,M}=\max_{x\in\mathbb{R}}a_{i}(x)$.
\begin{enumerate}[label=(B\arabic*)]
    \setcounter{enumi}{3}
    \item We assume there is a unique $\bar{x}_{1}$ satisfying $\bar{x}_{1}>x_{1}$ and $a(x_1)-\frac{c^2}{4}=a(\bar{x}_{1})$, and a unique $\bar{x}_{2}$ satisfying $\bar{x}_2<x_2$ and $a(x_2)-\frac{c_2^2}{4}=a(\bar{x}_2)$. \label{assum:B4}
\end{enumerate}
Since, in the absence of the other peak, the solution would concentrate on either $\bar{x}_1$ or $\bar{x}_2$, we  again refer to these points as the lagged optima.

This model represents mutation and competition in the presence of two alternative shifting optima in the fitness landscape. In other words, at a given time $t$, there are two locally optimum traits, $x_{1}+\varepsilon{c_1}t$ and $x_{2}+\varepsilon{c_2}t$. In light of the previous results, we expect to obtain conditions on $a_{i,MD}$ and $c_{i}$ that determine whether the population goes extinct or not, and where the population concentrates if it does not.

To state our main result for this case, we need the following transformed problems. In what follows, we suppress the $\varepsilon$ subscripts for ease of notation. We let $N_{i}(x,t)=n(\varepsilon{}c_{i}t+x,t)$, $C=c_{2}-c_{1}$ and $m_{i}=N_{i}(x,t)e^{\int_{0}^{t}\rho_{\varepsilon}(s)ds}$. Then $m_{i}$ solves: 
\begin{equation*}\label{eqn:TwoPeaksDifferentDirectionsLin}
\begin{cases*}
        \partial_{t}m_{i}-\varepsilon^2\partial_{xx}{m}_{i}-\varepsilon{}c_{i}\partial_{x}m_{i}=m_{i}\left(a_{i}(x)+a_{i'}(x+(-1)^{i}\varepsilon{C}t)-\delta\right) &$(x,t)\in{\mathbb{R}}\times{}\mathbb{R}^{+},$ \\
        m_{i}(x,0)=n_{0}(x).
\end{cases*}
\end{equation*}
where $i'$ is $1$ if $i=2$ and $2$ if $i=1$. Repeating similar transformations as before, by letting ${M}_{i}=m_{i}e^{\frac{c_{i}x}{2\varepsilon}}$, we find that ${M}_{i}$ solves
\begin{equation*}\label{eqn:TwoPeaksDifferentDirectionsLinLiouville}
\begin{cases*}
        \partial_{t}{M}_{i}-\varepsilon^2\partial_{xx}{{M}_{i}}={M}_{i}\left(a_{i}(x)+a_{i'}(x+(-1)^{i}\varepsilon{C}t)-\frac{c_{i}^2}{4}-\delta\right) &$(x,t)\in{\mathbb{R}}\times{}\mathbb{R}^{+},$ \\
        {M}_{i}(x,0)=n_{0}(x)e^{\frac{c_{i}x}{2\varepsilon}}.
\end{cases*}  
\end{equation*}
We also need the following eigenvalue problems,
\begin{equation}\label{eqn:TwoPeaksDifferentDirectionsLinEigen}
\begin{cases*}
        -\varepsilon^2\partial_{xx}{p}_{i}-\varepsilon_{}c_{i}\partial_{x}{p}_{i}={p}_{i}\left(\lambda_{i,\varepsilon}+a_{i}(x)-\delta\right) &$x\in{\mathbb{R}},$ \\
       p_{i,\varepsilon}>0 &    $x\in{\mathbb{R}},$
\end{cases*}  
\end{equation}
and
\begin{equation}\label{eqn:TwoPeaksDifferentDirectionsLinLiouvilleEigen}
\begin{cases*}
        -\varepsilon^2\partial_{xx}{p_{i,\varepsilon}}=p_{i,\varepsilon}\left(\lambda_{i,\varepsilon}+a_{i}(x)-\frac{c_{i}^2}{4}-\delta\right) &$x\in{\mathbb{R}},$ \\
       p_{i,\varepsilon}>0 &$x\in{\mathbb{R}}.$\\
\end{cases*}  
\end{equation}
We note that, once we fix a normalisation of $p_{i,\varepsilon}$, we have that $p_{i,\varepsilon}=w_\varepsilon{}p_{i,\varepsilon}e^{\frac{-c_ix}{2\varepsilon}}$ for some positive constant $w_\varepsilon$.

We proceed similarly to \cite{figueroa2018long}, and obtain estimates on local growth rate near each of the dominant traits. Our main result determines conditions on the survival of the population, in the long-time, rare mutation limit. When it does not go extinct, we show that it concentrates on the lagged optima with maximum  fitness.

\begin{theorem}\label{thm:TwoPeaksMainResult}
Under the assumptions \labelcref{assum:B1,assum:SignsOfShift,assum:B3,assum:B4}, we have that
\[p_{1,\varepsilon}(x)\xrightharpoonup[\varepsilon\rightarrow{0}]{}\delta(x-\bar{x}_1),~~\lambda_{1,\varepsilon}\xrightarrow[\varepsilon\rightarrow{0}]{}a_1(x_1)-\frac
{c_1^2}{4}-\delta.\] If $a_{1}(x_{1})-\frac{c_{1}^2}{4}>\max\{\delta,a_{2}(x_{2})-\frac{c_{2}^2}{4}\}$ then the population persists in the long-time limit; in fact
\[\left\Vert\frac{n_\varepsilon(x+\varepsilon{}c_{1}t,t)}{\rho_\varepsilon(t)}-p_{1,\varepsilon}(x)\right\Vert_{L^{\infty}(\mathbb{R})}\xrightarrow[t\rightarrow\infty]{}0,\]
and
   \[\rho_\varepsilon(t)\xrightarrow[t\rightarrow\infty]{}\int_{}a_1(y)p_{1,\varepsilon}(y)dy.\]

   If $a_{i}(x_{i})-\frac{c_{i}^2}{4}-\delta\leq{0}$ for $i=1,2$ then $n_\varepsilon$ instead converges uniformly to $0$.

\end{theorem}


To prove \cref{thm:TwoPeaksMainResult}, we first construct a supersolution of a related problem which vanishes on the support of $a_2(x-\varepsilon{}c_2t)$, which enables us to reduce back to the single peak case. To end this section, we will offer some heuristics in support of \cref{thm:TwoPeaksMainResult}. Suppose that, as might be expected from \cite{figueroa2018long}, we have obtained the local growth rates in a ball of radius $R$:
\[N_i(x,t)e^{\int_{0}^{t}\rho_{\varepsilon}(s)ds+\lambda_{i,\varepsilon}t}\xrightarrow[t\rightarrow\infty]{}\alpha_{i}{p}_{i,\varepsilon}~~\forall|x|<R.\]
When extinction does not occur, we expect that $\rho_{\varepsilon}(t)\xrightarrow[t\rightarrow\infty]{}\rho_\infty$ for some positive constant, and, of course that $0<\Vert{}n\Vert_{L^{\infty}(\mathbb{R})}<\infty$. These facts together mean the only possible value of $\rho_\infty$ is $\text{max}(-\lambda_{1,\varepsilon},-\lambda_{2})$, since any other value would mean extinction (or blow up).

\section{Proofs of main results}\label{sec:MainResults}

\subsection{Proof of \cref{thm:EigenvalueConvergence,thm:RescaleConvergence}}\label{subsec:Proof1}

For Case 1, \cref{thm:EigenvalueConvergence} is a consequence of \cref{lma:UniformR,lma:EigenvectorConcentration,lma:ConvergenceToEigenvalue,lma:rhoConvergence} below. From here we assume fix $p_\varepsilon$ as the solution to \eqref{eqn:MoreGeneralProblemShift} such that $\Vert{p_\varepsilon}\Vert_{L^{1}(\mathbb{R})}=1$.

\begin{lemma}\label{lma:UniformR}
Let $\lambda_{R,\varepsilon}$ be the eigenvalue from \eqref{eqn:DirichletPrince}. Then $\lambda_{R,\varepsilon}\xrightarrow[\varepsilon\rightarrow{0}]{}a_{M}-\frac{c^2}{4}$ and this convergence is uniform in $R$.
\end{lemma}
\begin{lemma}\label{lma:EigenvectorConcentration}
    Assume $a_{M}-\frac{c^2}{4}>0$ and  \ref{assum:A1},\ref{assum:A2},\ref{assum:A4},\ref{assum:A5}, then 
    \[{p}_\varepsilon\xrightharpoonup[\varepsilon\rightarrow0]{}\sum_{}a_{i}\delta_{\bar{x}_{i}}.\]
    where the $a_{i}$ are non-negative and sum to $1$.
\end{lemma}

\begin{lemma}\label{lma:ConvergenceToEigenvalue}   Under the assumptions $a_{M}-\frac{c^2}{4}>0$, $\varepsilon$ is small enough, and \ref{assum:A1}-\ref{assum:A5} the normalized population will converge to  $p_\varepsilon$ as $t\rightarrow\infty$, i.e
         \[\left\Vert{\frac{N_\varepsilon}{\rho_{\varepsilon}(t)}-{p}_{\varepsilon}}\right\Vert_{L_{\infty}(\mathbb{R})}\xrightarrow[t\rightarrow\infty]{}0.\]
\end{lemma}

\begin{lemma}\label{lma:rhoConvergence} Under the assumptions $a_{M}-\frac{c^2}{4}>0$, $\varepsilon$ is small enough, and \ref{assum:A1}-\ref{assum:A5} the total population $\rho_\varepsilon(t)$ will convergence to a finite, positive value
    \[\rho_\varepsilon(t)\xrightarrow[t\rightarrow\infty]{}\int_{\mathbb{R}}a(y)p_\varepsilon(y)dy.\]
\end{lemma}

\cref{thm:RescaleConvergence} will be proved using \cref{lma:TopologicalPressure} (in \cref{subsec:Prelim}) and \cref{lma:LocUniEigenvalueConvergence2} below.

\begin{lemma}\label{lma:LocUniEigenvalueConvergence2}
    The solutions $P_{R,\varepsilon}$ to \eqref{eqn:DirichletPrince} converge locally uniformly in $\mathbb{R}$ to a solution $P_\varepsilon$ to \eqref{eqn:MoreGeneralProblem}, and $\lambda_\varepsilon\xrightarrow[\varepsilon\rightarrow{0}]{}a_{M}-\frac{c^2}{4}.$ 
\end{lemma}

    

We begin with the proof of \cref{lma:UniformR} since we will require the limiting value of $\lambda_\varepsilon$ as $\varepsilon\rightarrow{0}$.

\begin{proof}[Proof of \cref{lma:UniformR}]

It is clear that $\lambda_{R,\varepsilon}\geq{-\left(a_{M}-\frac{c^2}{4}\right)}$ using the Rayleigh-Quotient:
\[\lambda_{R,\varepsilon}=\inf_{\phi\in{H}^{1}_{0}(B_{R})\backslash{0}}\frac{\varepsilon^2\int_{B_{R}}|\partial_{x}\phi|^2-\int_{B_{R}}\left(a(x)-\frac{c^2}{4}\right)\phi^2}{\int_{B_{R}}\phi^2}.\]
 We need to pick a sequence of $\phi_\varepsilon\in{H^{1}_{0}}(B_{R})\backslash\{0\}$ such that \[\frac{\varepsilon^2\int_{B_{R}}|\partial_{x}\phi_\varepsilon|^2-\int_{B_{R}}\left(a(x)-\frac{c^2}{4}\right)\phi_\varepsilon^2}{\int_{B_{R}}\phi_\varepsilon^2}\xrightarrow[\varepsilon\rightarrow{0}]{}-\left(a_{M}-\frac{c^2}{4}\right).\] We cannot use the same sequence as in \cite{lorenzi2020asymptotic} since their functions do not vanish on the boundary, and are hence not in $H_{0}^{1}(B_{R}).$ Instead we do the following. Let $\chi$ be a smooth cut off function such that $\chi(x)=0$ for $|x|>1$, and $\chi(x)=1$ for $x\in(-\frac{1}{2},\frac{1}{2})$.  Then, similarly to \cite{lorenzi2020asymptotic}, take $G=D_{1}\chi^2{}e^{-|x|^2}$ where $D_{1}$ is a normalising constant.

Define the sequence $\phi_\varepsilon^2=\frac{1}{\varepsilon^\frac{1}{2}}G\left(\frac{x-x_{i}}{\varepsilon^{\frac{1}{2}}}\right)$ for $x_{i}\in{}\text{argmax}_{x\in\mathbb{R}}\left(a(x)-\frac{c^2}{4}\right)$.  We have (due to the normalisation) that $\phi_\varepsilon^2\xrightharpoonup[\varepsilon{\rightarrow{0}}]{}\delta_{x_{i}}$. The choice of $x_{i}$ is arbitrary, so it so sufficient to find a sequence of functions which concentrate on just one of the global optima.  Thus, we only need to check that $\varepsilon^2\int_{B_{R}}|\partial_{y}\phi_{\varepsilon}(y)|^2dy\xrightarrow[\varepsilon{\rightarrow{0}}]{}0$. We compute
\[|\partial_{y}\phi_{\varepsilon}(y)|^2=\varepsilon^{-\frac{3}{2}}\left\vert(G^{\frac{1}{2}})'\left(\frac{y-x_{i}}{\varepsilon^{\frac{1}{2}}}\right)\right\vert^2,\]
which is sufficient since 
\begin{align*}
    \int_{B_{R}}\left\vert(G^{\frac{1}{2}})'\left(\frac{y-x_{i}}{\varepsilon^{\frac{1}{2}}}\right)\right\vert^2dy&\leq{}    \int_{\mathbb{R}}\left\vert(G^{\frac{1}{2}})'\left(\frac{y-x_{i}}{\varepsilon^{\frac{1}{2}}}\right)\right\vert^2dy\\
    &\leq{}\varepsilon^{\frac{1}{2}}\int_{\mathbb{R}}|(G^{\frac{1}{2}})'(y)|^2dy
\end{align*} and $\int_{\mathbb{R}}|(G^{\frac{1}{2}})'(y)|^2dy$ is a fixed constant.

Hence we get that $\lambda_{R,\varepsilon}\xrightarrow[]{}-\left(a_{M}-\frac{c^2}{4}\right)$. We note that we can pick the same $\phi_\varepsilon^2$ independently of $R$ (supposing $R$ is large enough) and so this convergence is  independent of $R$.
\end{proof}

We next prove the following lemma which shows that the eigenfunctions can be normalised.
\begin{lemma}\label{lma:DecayEigenvector}
    Assume that $\lambda_\varepsilon<0$, and let $p_\varepsilon^{\infty}$ be the solution to \eqref{eqn:MoreGeneralProblemShift} such that $\Vert{}p_\varepsilon^\infty\Vert_{L^\infty(\mathbb{R})}=1$.
    We have the following bounds:
    \[e^{-\frac{\underline{\kappa}|x-x_\varepsilon|}{\varepsilon}}\leq{}p_{\varepsilon}^{\infty}(x)\leq{}\min\left\{1,e^{\frac{-\overline{\kappa}(|x|-R_0)}{\varepsilon}}\right\},~~x\in\mathbb{R},\]
    where $\underline{\kappa}=\frac{c}{2}$, and $\overline{\kappa}=\frac{-c+\sqrt{2c^2+4\delta}}{2}$, and $x_\varepsilon$ is a point where $p_\varepsilon(x_\varepsilon)=1$. 
\end{lemma}

\begin{proof}[Proof of \cref{lma:DecayEigenvector}]
    We let $\mathcal{L}(u):=-\varepsilon^2\partial_{xx}u-c\varepsilon{}\partial_{x}u-(a(x)+\lambda_\varepsilon)u$. The first inequality is a  consequence the comparison theorem in \cref{sec:Preliminaryresults} that applies once we show $e^{\frac{-\overline{\kappa}(|x|-R_0)}{\varepsilon}}$ is a generalised supersolution of \eqref{eqn:MoreGeneralProblemShift} on $\mathbb{R}/B_{R_0}$. Since, $p_{\varepsilon}^\infty(\pm{}R_0)\leq{}1$ the boundary condition is satisfied and we only need to check that $\mathcal{L}\left(e^{\frac{-\overline{\kappa}(|x|-R_0)}{\varepsilon}}\right)\geq{0}$ for $|x|>R_0$. This reduces to $-\overline{\kappa}^2-c\overline{\kappa}-(-\delta+\lambda_\varepsilon)\geq{0}$, and is satisfied for the given $\overline{\kappa}$ since $\lambda_\varepsilon\leq{0}$. 
    
    The lower bound is similar, except the sufficient condition is $-\underline{\kappa}^2+c\underline{\kappa}-(a_M+\lambda_\varepsilon)\leq{}0$. One checks that for $\underline{\kappa}=\frac{c}{2}$ this reduces to $\frac{c^2}{4}-(a_M+\lambda_\varepsilon)\leq{0}$ which is satisfied due to the bounds on $\lambda_\varepsilon$ obtained in the proof of \cref{lma:UniformR}. 

\end{proof}

\begin{remark}
    Since $\overline{\kappa}>\frac{c}{2}$ this implies that $p_\varepsilon(x)e^{\frac{cx}{2\varepsilon}}$ is integrable. 
\end{remark}

We next prove \cref{lma:LocUniEigenvalueConvergence2} which allows us to use $p_{R,\varepsilon}$ to estimate $p_\varepsilon$



\begin{proof}[Proof of \cref{lma:LocUniEigenvalueConvergence2}]
    Let $\tilde{p}_{R,\varepsilon}=e^{-\lambda_{R,\varepsilon}t}P_{R,\varepsilon}$. which solves:
\begin{equation}\label{eqn:NonEigen}
   \begin{cases*}
        \partial_{t}\tilde{p}_{R,\varepsilon}-\varepsilon^2\partial_{xx}\tilde{p}_{R,\varepsilon}-\left(a(x)-\frac{c^2}{4}\right)\tilde{p}_{R,\varepsilon}=0 &\text{ in } $B_{R},$ \\ \tilde{p}>0, & \\
        \tilde{p}=0 &\text{ on } $\partial{}B_{R}.$
        \end{cases*}   
\end{equation}
This is a particular case of the sequence of Dirichlet problems considered in \cite{huska2008exponential}. We only need to check the exponential separation property, which follows straightforwardly from the formula $\tilde{p}_{R,\varepsilon}=e^{-\lambda_{R,\varepsilon}t}P_{R,\varepsilon}$. Hence, by \cref{lma:LocUniEigenvalueConvergence}, a subsequence $\tilde{p}_{R_{n},\varepsilon}$ converges locally uniformly in $\mathbb{R}\times\mathbb{R}$ to  $\tilde{p}_\varepsilon$ which solves
\begin{equation}\label{eqn:MoreGeneralProblemNonEigen}
   \begin{cases*}
        \partial_{t}\tilde{p}_\varepsilon-\varepsilon^2\partial_{xx}\tilde{p}_\varepsilon-\left(a(x)-\frac{c^2}{4}\right)\tilde{p}_\varepsilon=0 &\text{ in } $\mathbb{R}\times{}\mathbb{R}$ \\ p>0, & {}
                 \end{cases*}   
\end{equation}
Since  $\lambda_{R,\varepsilon}$ are monotonic in $R$ via the Rayleigh formula, and bounded because of \cref{lma:UniformR}, we know these converge to a limiting value $\lambda_\varepsilon$, and that $\lambda_\varepsilon\xrightarrow[\varepsilon\rightarrow{0}]{}{-\left(a_{M}-\frac{c^2}{4}\right)}$.  Thus:

\[P_{R,\varepsilon}=e^{\lambda_{R,\varepsilon}{t}}\tilde{p}_{R,\varepsilon}\xrightarrow[R\rightarrow\infty]{}e^{\lambda_{\varepsilon}t}\tilde{p}_\varepsilon.\] 

One checks that the last term is $P_{\varepsilon}$  which the exact solution to \eqref{eqn:MoreGeneralProblem}.
\end{proof}
Using this result, we can prove \cref{thm:RescaleConvergence}.

\begin{proof}[Proof of \cref{thm:RescaleConvergence}]

Note that the expression
 \[\hat{p}_\varepsilon=\frac{{p}_\varepsilon{}e^{\frac{cx}{2\varepsilon}}}{\Vert{}{p}_\varepsilon{}e^{\frac{cx}{2\varepsilon}}\Vert_{L_{1}(\mathbb{R})}},\]
 is merely the $L_{1}$ normalised solution $P_\varepsilon$ to \eqref{eqn:MoreGeneralProblem}. According to \cref{lma:TopologicalPressure}, the functions $P_{R,\varepsilon}$, similarly normalized in $L^{1}$, satisfy:
   \[P_{R,\varepsilon}\xrightharpoonup[\varepsilon\rightarrow{0}]{}\sum_{x_{i}\in{M}\cap{M_{1}}}a_{i}\delta_{x_{i}}.\]
   But $P_{R,\varepsilon}$ approaches $P_{\varepsilon}$ locally uniformly according to \cref{lma:LocUniEigenvalueConvergence2}. Hence if $P_{R,\varepsilon}\rightarrow{0}$ as $\varepsilon\rightarrow{0}$ in some bounded open set $U$, then the same is true for $P_{\varepsilon}$. This completes the proof.
 
\end{proof}

Next we proceed with the proof of \cref{thm:EigenvalueConvergence}, but require some more set up to do so.
 Our approach here is similar to that in \cite{iglesias2021selection} except that we are working directly with viscosity solutions to the (time-independent) eigenvalue problem rather than with the original PDE. We will first apply the WKB ansatz. Letting $\psi_\varepsilon$ satisfy,
\[{p}_{\varepsilon}=e^{\frac{\psi_{\varepsilon}}{\varepsilon}}.\]
We find that $\psi_{\varepsilon}$ solves 
\[
  -\varepsilon\partial_{xx}\psi_\varepsilon-\left\vert\partial_{x}\psi_{\varepsilon}+\frac{c}{2}\right\vert^2-a(x)+\frac{c^2}{4}-\lambda_{\varepsilon}=0.\]
We will show, analogously to Theorem 1.1 in \cite{iglesias2021selection}, the following lemma

\begin{lemma}\label{lma:ViscocitySoln}
    Assuming \labelcref{assum:A1,assum:A2,assum:A4,assum:A5} the function $\psi_\varepsilon$ converges (up to extraction of subsequences) locally uniformly to a viscosity solution $\psi$ of 

\begin{equation}\label{eqn:HJ_Adv}
\begin{cases}
   -\left\vert\partial_{x}\psi+\frac{c}{2}\right\vert^2+a_M-a(x)=0,\\
    \max_{x\in\mathbb{R}}\psi=0.\\
\end{cases}    
\end{equation}

\end{lemma}
We defer the proof to the appendix for completeness. 

A corollary of \cref{lma:ViscocitySoln} is that ${p}_\varepsilon$ concentrates on the set of points such that $\psi=0$.

\begin{cor}
    Assume there are only finitely many points solving $\psi(x)=0$. Let these be $x_{1}',x_{2}',...,x_{k}'$. For the function ${p}_\varepsilon$ we have, after an extraction of subsequences

    \[{p}_\varepsilon\xrightharpoonup[\varepsilon\rightarrow{0}]{}\sum_{i=1}^k{}a_{i}\delta_{x'_{i}},\]

    where $a_{i}\geq{}0$ for each $i$ and $\sum_{i}a_{i}=1$.
\end{cor}
Note that we can also prove that ${p}_\varepsilon$ concentrates on the finitely many points by following Proposition 1 in \cite{lorenzi2020asymptotic}. Unfortunately, this method includes more concentration points than expected, since it does not exclude the smallest value $y_{i}$ which satisfies both $y_{i}>x_{i}$ and $a(y_{i})=a_{M}-\frac{c^2}{4}$. We can get a more refined result using the Hamilton-Jacobi equation method. Thus we consider the equation solved by $u=\psi(x)+\frac{cx}{2}$, which is:
\begin{equation}\label{eqn:HJV}
\begin{cases}
    -\left\vert\partial_{x}u\right\vert^2+a_M-a(x)=0,\\
    \max_{x\in\mathbb{R}}u-\frac{cx}{2}=0.\\
\end{cases}
\end{equation}
We now find the possible set of points for which $\psi=0$. We do not aim to show uniqueness of solutions, only refine the set of concentration points. It is straightforward to verify that viscosity solutions of \eqref{eqn:HJV} are also visocity solutions of 
\begin{equation}\label{eqn:HJV2}
\begin{cases}
    -\left\vert\partial_{x}u\right\vert+\sqrt{a_M-a(x)}=0,\\
    \max_{x\in\mathbb{R}}u-\frac{cx}{2}=0.\\
\end{cases}
\end{equation}
We recall a well-known result for viscosity solutions of Hamilton-Jacobi equations which is proved in \cite{PLLionsToApprox1984}. Let $\Omega\subset{\mathbb{R}}$ be a bounded domain and $n(x)\in{C}\left(\bar{\Omega}\right)$ and $n>0$. Consider the equation:
\begin{equation}\label{eqn:HJE}
\left\{\begin{array}{l}
        |\partial_{x}\tilde{u}|=n(x) \text{ in } \Omega. \\
        \tilde{u}(x)=\phi(x) \text{ in } \partial\Omega,
        \end{array}
\right.  
\end{equation}
We let
\begin{equation*}
    \begin{split}
    L(x,y)=&\inf\left\{\int_{0}^{T_{0}}n(\zeta(s))ds : (T_{0},\zeta) \text{ such that } \right. \\
    &\left. \zeta(0)=x,\zeta(T_{0})=y,\left\vert\frac{d\zeta}{ds}\right\vert\leq{1} \text{ a.e in 
 } [0,T_{0}], \zeta(t)\in\bar{\Omega} \quad \forall{t}\in[0,T_{0}] \right\}. 
    \end{split}
\end{equation*}
One has a representation formula for solutions in terms of $L$.

\begin{lemma}[\cite{PLLionsToApprox1984}]\label{lma:SupHJS}
    The viscosity solution to \eqref{eqn:HJE} is unique and given by \[\tilde{u}(x)=\inf_{y\in\partial\Omega}[\phi(y)+L(x,y)].\]
\end{lemma}
It is straightforward to verify that if $\tilde{u}$ solves \eqref{eqn:HJE}, then $u=-\tilde{u}$ solves
\begin{equation}\label{eqn:HJE2}
\left\{\begin{array}{l}
        -|\partial_{x}u|=-n(x) \text{ in } \Omega, \\
        u(x)=-\phi(x) \text{ in } \partial\Omega,
        \end{array}
\right.  
\end{equation}
and thus the solution to \eqref{eqn:HJV} is
\[u=\sup_{y\in\partial\Omega}[\phi'(y)+L'(x,y)],\]
where $\phi'$ prescribes the boundary value and
    \begin{equation*}
    \begin{split}
    L'(x,y)=&\sup\left\{-\int_{0}^{T_{0}}\sqrt{a_M-a(x)}ds : (T_{0},\zeta) \text{ such that } \right. \\
    &\left. \zeta(0)=x,\zeta(T_{0})=y,\left\vert\frac{d\zeta}{ds}\right\vert\leq{1} \text{ a.e in 
 } [0,T_{0}], \zeta(t)\in\bar{\Omega} \quad \forall{t}\in[0,T_{0}] \right\}. 
    \end{split}
\end{equation*}
When a general $n(x)$ in \eqref{eqn:HJE} vanishes at multiple points in $\Omega$, then one cannot guarantee a unique viscosity solution, as is the case for $n(x)=\sqrt{a_{M}-a(x)}$. However, the representation formula still applies between the zeros of $n(x)$ and this is enough to refine the set of concentration points.

We are now ready to prove \cref{lma:EigenvectorConcentration} which is the first part of \cref{thm:EigenvalueConvergence}.

\begin{proof}[Proof of \cref{lma:EigenvectorConcentration}]
We apply this representation formula between two maxima $x_{1}$ and $x_{2}$ to get
\[u(x)=\max\left\{f_1(x);f_2(x)\right\},\]
where $f_1(x)=u(x_{1})-\int_{x_{1}}^{x}\sqrt{a_{M}-a(y)}dy$ and $f_2(x)=u(x_{2})-\int_{x}^{x_{2}}\sqrt{a_{M}-a(y)}dy$.
We let $x^{*}$ be an intersection point between $u(x)$ and $\frac{cx}{2}$, and  define $x_{1}^{+}$  as the minimum solution to $a(x)=a_{M}-\frac{c^2}{4}$ satisfying $x>x_1$ and $x_{2}^{-}$ as the maximum solution to the same equation satisfying $x_{2}^{-}<x_2$. These are the only solutions on the interval $(x_1,x_2)$ according to (A5). The intersection of $u$ with $\frac{cx}{2}$ cannot be at the point $z$ where $f_1(z)=f_2(z)$ since $f_1$ is a decreasing function so would have to intersect from above, contradicting the constraint.

Therefore, $x^{*}\in{}(z,x_2]$ and we have, in this interval,
\[u(x)=u(x_{2})-\int_{x}^{x_{2}}\sqrt{a_{M}-a(y)}dy\]
and 
\begin{equation}
\partial_{x}u=\sqrt{a_{M}-a(x)}.\label{eqn:DerivativeU}
\end{equation}
By the above expression for the gradient, assumption (A5), and the continuity of $a(x)$, it follows that  $\partial_{x}u<\frac{c}{2}$ on $[x_1,x_1^{+})\cup(x_2^{-},x_2]$ we must also have $x^{*}\in[x_1^{+},x_2^{-}]\cap(z,x_2^{-}]$. The intersection $x^{*}$ is therefore not at the end point, and so $u$ is tangent to $\frac{cx}{2}$ at the intersection, thus  $x^{*}\in\{x_{1}^{+},x_{2}^{-}\}$.
We must also have that $\partial_{xx}u(x^{*})<0$ since otherwise $u$ will exceed $\frac{cx}{2}$. Differentiating \eqref{eqn:DerivativeU} we find
\[\partial_{xx}u=-\frac{1}{2}\frac{a'(x)}{\sqrt{a_M-a(x)}}.\]
By the definition of the points $x_{1}^{+}$ and $x_{2}^{-}$ we have that $a'(x_{1}^{+})<0$ and $a'(x_{2}^{-})>0$, therefore only $x_{2}^{-}$ is a possible concentration point.
    
\end{proof}

We can now relate this to the solution $n_\varepsilon(x,t)$ of \eqref{eqn:UnboundedShifting}, which we do with \cref{lma:ConvergenceToEigenvalue}. Firstly, we will need the following lemma.

\begin{lemma}\label{lma:SuperSoln}
Assume $\lambda_\varepsilon<0$.
Let 
$W_\varepsilon(x,t)=N_\varepsilon(x,t)e^{\int_{0}^{t}\rho_{\varepsilon}(s)ds+\lambda_{\varepsilon}t}$ which solves
\begin{equation}\label{eqn:AuxEqn}
    \begin{cases*}
        \partial_{t}W_\varepsilon-\mathcal{L}W_\varepsilon=0,&~~$(x,t)\in{}\mathbb{R}\times(0,\infty)$\\
        W_\varepsilon(x,0)=n_0(x),&~~$x\in\mathbb{R}$,
    \end{cases*}
\end{equation}
where $\mathcal{L}w:=c\varepsilon\partial_{x}w+\varepsilon^2\partial_{xx}w+(a(x)+\lambda_\varepsilon)w$.  
Then there exists an integrable function $\overline{W}_\varepsilon(x)$ such that 
\[W_\varepsilon(x,t)\leq{}\overline{W}_\varepsilon(x),~~\forall{}x\in\mathbb{R}\]

\end{lemma}

\begin{proof}

Let $\eta$ be a small constant. Then $u_\eta(x,t)=W_\varepsilon(x,t)e^{\eta{t}}$ solves 
    \[\partial_{t}u_\eta-\varepsilon{}c\partial_{x}u_\eta-\varepsilon^2\partial_{xx}u_\eta=u_\eta\left(a(x)+\lambda_\varepsilon+\eta\right).\]
    The function $\phi_\eta(x,t)=e^{\eta{t}}p_\varepsilon(x)$ is a positive entire solution of this differential equation that  satisfies hypothesis (C2) as in \cref{subsec:Prelim}. We thus apply Corollary 4 and find
    \[\frac{\Vert{}u_\eta(.,t)-\alpha_\varepsilon\phi_\eta(.,t)\Vert_{L^{\infty}(\mathbb{R})}}{\Vert{}\phi_\eta(.,t)\Vert_{L^{\infty}(\mathbb{R})}}\xrightarrow[t\rightarrow\infty]{}0.\]
    The $e^{\eta{t}}$ terms cancel, so that $\Vert{}W_\varepsilon(x,t)-\alpha_\varepsilon{}p_{1,\varepsilon}(x)\Vert_{L^\infty(\mathbb{R})}\xrightarrow[t\rightarrow\infty]{}0$. Thus there is a $T_1$ such that $W(x,t)<2\alpha_\varepsilon{}p_\varepsilon(x)$ for all $(x,t)\in{}[-R_0,R_0]\times{}[T_1,\infty)$. 
   We next claim that $W_\varepsilon(x,t)\leq{}\tilde{W}_\varepsilon(x,t)=e^{C_1-C_2|x|+(a_{M}+\lambda_\varepsilon)t}$. Substituting $\tilde{W}_\varepsilon(x,t)$ into the differential equation gives
   \begin{align*}
    \partial_{t}\tilde{W}_\varepsilon-\mathcal{L}\tilde{W}_\varepsilon\geq{}-\varepsilon^{2}C_2^2+\varepsilon{}C_2-(a(x)-a_{M}),
\end{align*}
    which is positive, for small enough $\varepsilon$, by the definition of $a_{M}$ as the maximum of $a(x)$. Since $n_0(x)\leq{}e^{C_1-C_2|x|}$ by assumption (A3), the comparison theorem yields the desired result.
    
We define $\overline{W}(x)$ as
\[\overline{W}_\varepsilon(x)=\begin{cases}
    e^{C_3+C_1-C_2|x|+(a_{1,M}+\lambda_\varepsilon)T_1}~~&x<-R_0,\\
    2\alpha_\varepsilon{}p_\varepsilon(x)~~&|x|<R_0,\\
       e^{C_4+C_1-C_2|x|+(a_{1,M}+\lambda_\varepsilon)T_1}~~&x>R_0,
\end{cases}\]
where we pick $C_3$ and $C_4$ to ensure that $\overline{W}_\varepsilon(\pm{}R_0)=2\alpha_\varepsilon{}p_\varepsilon(\pm{}R_0)$. By construction $\overline{W}_\varepsilon$ dominates $W_\varepsilon(x,t)$ on the parabolic boundary of $((-\infty,-R_0)\cup{}(R_0,\infty))\times{}(T_1,\infty)$. We also check that 
\begin{align*}
    \partial_{t}\overline{W}_\varepsilon-\mathcal{L}\overline{W}_\varepsilon\geq{}-\varepsilon^{2}C_2^2+\varepsilon{}C_2-(-\delta+\lambda_\varepsilon),
\end{align*}
which is positive for $\varepsilon$ small enough. Thus by an application of the comparison principle, $W_\varepsilon(x,t)\leq{}\overline{W}_\varepsilon(x,t)$ for $(x,t)\in((-\infty,-R_0)\cup{}(R_0,\infty))\times{}(0,\infty)$. In particular $W_\varepsilon(x,t)$ is bounded by a constant in time integrable function.

\end{proof}

We can now proceed with the proof of \cref{lma:ConvergenceToEigenvalue}.

\begin{proof}[Proof of \cref{lma:ConvergenceToEigenvalue}]
The first part follows similarly to the proof of proposition 2 in \cite{figueroa2018long}. We again let $\eta$ be a small constant, and now let $u_\eta(x,t)=N_\varepsilon(x,t)e^{\frac{cx}{2\varepsilon}+\int_{0}^{t}\rho_{\varepsilon}(s)ds+\lambda_{\varepsilon}t+\eta{t}}$ which solves 
    \[\partial_{t}u_\eta-\varepsilon^2\partial_{xx}u_\eta=u_\eta\left(a(x)-\frac{c^2}{4}+\lambda_\varepsilon+\eta\right).\]
    The function $\phi_\eta(x,t)=e^{\eta{t}}P_\varepsilon(x)$ is a positive entire solution to the differential equation (without initial conditions) that  satisfies hypothesis (C2) as in \cref{subsec:Prelim}. We thus apply Corollary 4 and find:
    \[\frac{\Vert{}u_\eta(.,t)-\alpha_\varepsilon\phi_\eta(.,t)\Vert_{L^{\infty}(\mathbb{R})}}{\Vert{}\phi_\eta(.,t)\Vert_{L^{\infty}(\mathbb{R})}}\xrightarrow[t\rightarrow\infty]{}0.\]
    The $e^{\eta{t}}$ terms cancel  so we obtain:
\[\Vert{}N_\varepsilon(x,t)e^{\frac{cx}{2\varepsilon}x+\int_{0}^{t}\rho_{\varepsilon}(s)ds+\lambda_{\varepsilon}t}-\alpha_{\varepsilon}P_\varepsilon{}\Vert_{L^{\infty}(\mathbb{R})}\xrightarrow[t\rightarrow{\infty}]{}0.\]
    Moreover, this convergence is exponential. We recall $P_\varepsilon=p_\varepsilon{}e^\frac{cx}{2\varepsilon}$ and rewrite in terms ${p}_\varepsilon$:
    \[\Vert{}N_\varepsilon(x,t)e^{\frac{cx}{2\varepsilon}+\int_{0}^{t}\rho_{\varepsilon}(s)ds+\lambda_{\varepsilon}t}-\alpha_{\varepsilon}{p}_\varepsilon{}e^{\frac{cx}{2\varepsilon}}\Vert_{L^{\infty}(\mathbb{R})}\xrightarrow[t\rightarrow{\infty}]{}0.\]
    This implies 
     \begin{equation}\label{eqn:Sigma_1Def}
N_\varepsilon(x,t)e^{\int_{0}^{t}\rho_{\varepsilon}(s)ds+\lambda_{\varepsilon}t}=\alpha_{\varepsilon}{p}_\varepsilon{}+\Sigma_{1}(x,t)e^{-\frac{cx}{2\varepsilon}},
     \end{equation}
    where $\Vert{}\Sigma_{1}(x,t)\Vert_{L^{\infty}(\mathbb{R})}\xrightarrow[t\rightarrow{\infty}]{}0$ exponentially.
    We use this to write
    \[\frac{N_\varepsilon}{\rho_{\varepsilon}}=\frac{\alpha_\varepsilon{}{p}_\varepsilon+\Sigma_{1}(x,t)e^{-\frac{cx}{2\varepsilon}}}{\int{}\alpha_\varepsilon{}{p}_\varepsilon{}+\Sigma_{1}(y,t)e^{-\frac{cx}{2\varepsilon}}dy},\] 
    and
    \[\rho_{\varepsilon}(t)=\int_{\mathbb{R}}\alpha_\varepsilon{}{p}_\varepsilon{}dy+\int_{\mathbb{R}}\Sigma_{1}(y,t)e^{-\frac{cx}{2\varepsilon}}dy.\]
    We need to show that the latter term converges to $0$ as $t\rightarrow{\infty}$. By  \cref{lma:SuperSoln} and \eqref{eqn:Sigma_1Def}
    \[\Sigma_{1}(x,t)e^{-\frac{cx}{2\varepsilon}}\leq{}\overline{W}_\varepsilon(x)+\alpha_\varepsilon{}p_\varepsilon,\]
    and each term on the right hand side is integrable. Thus by the Dominated Convergence Theorem, using the exponential decay of $\Sigma_1(x,t)$ in time, we have $\int_{\mathbb{R}}\Sigma_1(x,t)e^{-\frac{cx}{2\varepsilon}}\xrightarrow[t\rightarrow{\infty}]{}0$.
    








    This completes the proof.
\end{proof}

We can now prove \cref{lma:rhoConvergence}.

\begin{proof}[Proof of \cref{lma:rhoConvergence}]
This follows similarly to \cite{figueroa2018long}. Computations identical to the preceding lemma show that
\[\left\Vert{}\frac{\int_{}a(y)N_\varepsilon(y,t)dy}{\rho_{\varepsilon}(t)}-\int_{}a(y){p}_\varepsilon(y)dy\right\Vert_{L^\infty(\mathbb{R})}\xrightarrow[t\rightarrow{\infty}]{}0.\]
This convergence is exponential. 
By integrating the original equation, we obtain:
\[\frac{d\rho_{\varepsilon}}{dt}=\rho_{\varepsilon}\left(\int_{}a(y){p}_\varepsilon(y){dy}+\Sigma_{2}(t)-\rho_{\varepsilon}\right),\]
where $\Sigma_{2}(t)$ is exponentially decreasing. By applying Gronwall's lemma, we can show the long term limit of $\rho_{\varepsilon}$ is $\int_{}a(y){p}_\varepsilon(y)$. This can be done as follows:

Firstly, we claim $\rho_{\varepsilon}(t)$ is eventually bounded above by $\overline{\gamma}=\int_{}a(y){p}_\varepsilon(y)dy+2|\Sigma_{2}(T)|$ for any fixed $T$. Suppose that $\rho_{\varepsilon}(t)$ exceeds $\overline{\gamma}$ at some point $t_{1}>T$ and let $t_{2}$ be next time where $\rho_{\varepsilon}(t)=\overline{\gamma}$ if such a point exists and be $\infty$ otherwise. Then for all $t\in(t_{1},t_{2})$
\begin{align*}
    \frac{d\rho_{\varepsilon}}{dt}&\leq{}\rho_{\varepsilon}\left(\Sigma_{2}(t)-2|\Sigma_{2}(T)|\right)\\
    &\leq{}-\rho_{\varepsilon}|\Sigma_{2}(T)|.
\end{align*}
Thus by Gronwall's inequality, 
\[\rho_{\varepsilon}(t)\leq{}\rho_{\varepsilon}(t_{1})e^{-(t-t_{1})|\Sigma_2(T)|}.\]
This implies $t_{2}$ is finite. But now it is clear $\rho_{\varepsilon}'(t)$ will be negative if $\rho_{\varepsilon}$ exceeds $\overline{\gamma}$ again, thus $\overline{\gamma}$ is an upper bound for all $t>t_{1}$.
Since this is true for any $T>0$ this shows:
\[\limsup_{t\rightarrow\infty}\rho_{\varepsilon}(t)=\int_{}a(y){p}_\varepsilon(y){dy}.\]

We must still find a lower bound. Firstly, since $N\geq{0}$ we can ensure $0$
is a lower bound for $\rho_{\varepsilon}(t)$. 
We define $\hat{\rho_{\varepsilon}}(t)=\rho_{\varepsilon}(t)e^{-\int_{0}^{t}\Sigma_2(s)ds}$ which solves:
\[\frac{d\hat{\rho_{\varepsilon}}}{dt}=\hat{\rho_{\varepsilon}}\left(\int_{}a(y){p}_\varepsilon(y){dy}-\hat{\rho_{\varepsilon}}e^{\int_{0}^{t}\Sigma(s)ds}\right).\]
The term $e^{\int_{0}^{t}\Sigma_2(s)ds}$ is bounded above by $k=e^{\int_{0}^{\infty}|\Sigma_2(s)|ds}$ so we have:
\[\frac{d\hat{\rho_{\varepsilon}}}{dt}\geq{}\hat{\rho_{\varepsilon}}\left(\int_{}a(y){p}_\varepsilon(y){dy}-k\hat{\rho_{\varepsilon}}\right).\]
This means that $\frac{d\hat{\rho_{\varepsilon}}}{dt}$ is positive as long as $\hat{\rho_{\varepsilon}}\leq{}\frac{\int_{}a(y){p}_\varepsilon(y){dy}{}}{k}$. In particular this means $\hat{\rho_{\varepsilon}}$ is increasing and thus $\rho_{\varepsilon}$ has:
\[\liminf_{t\rightarrow\infty}\rho_{\varepsilon}\geq{}\rho_\text{min}>0.\] 
The lower bound is important in what follows. 
Proceeding similarly as before we define $\underline{\gamma}=\int_{}a(y){p}_\varepsilon(y){dy}-2|\Sigma_{2}(T)|$ for an arbitrary $T$. We suppose there is a time $t_{3}>T$ such that $\int_{}a(y){p}_\varepsilon(y)+\Sigma_{2}(t_{3})>0$ and $\frac{\rho_\text{min}}{2}<\rho_{\varepsilon}(t)<\underline{\gamma}$ for all $t\in(t_3,t_4)$ where $t_4$ is the next time such that $\rho_\varepsilon=\underline{\gamma}$ or $\infty$ if no such time exists. 

Then:
\[\frac{d\rho_{\varepsilon}}{dt}\geq{}\frac{\rho_\text{min}}{2}\left(\int{}a(x)p_\varepsilon-|\Sigma_2(T)|-\rho_{\varepsilon}\right).\]
Let $\tilde{\rho_{\varepsilon}}=\int_{}a(y){p}_\varepsilon(y){dy}-|\Sigma_2(T)|-\rho_{\varepsilon}$. Then $\tilde{\rho_{\varepsilon}}$ satisfies:
\[\frac{d\tilde{\rho_{\varepsilon}}}{dt}\leq{}-\frac{\rho_\text{min}}{2}\tilde{\rho_{\varepsilon}}.\]
Similarly to before, $\tilde{p}_\varepsilon$ is decreasing (and thus $\rho_{\varepsilon}$ is increasing) at an exponential rate in this case and therefore $t_4$ is finite. But $\frac{d\rho_\varepsilon}{dt}(t)>0$ for any $t>t_4$ such that $\rho_\varepsilon(t)=\underline{\gamma}$ and
so cannot decrease below $\underline{\gamma}$ again. This gives:
\[\liminf_{t\rightarrow\infty}\rho_{\varepsilon}(t)=\int_{}a(y){p}_\varepsilon(y){dy}.\]
\end{proof}

\begin{proof}[Proof of \cref{thm:EigenvalueConvergence}]
    \cref{thm:EigenvalueConvergence} follows by combining  \cref{lma:EigenvectorConcentration,lma:ConvergenceToEigenvalue,lma:rhoConvergence}.
\end{proof}

\subsection{Proof of \cref{thm:TwoPeaksMainResult}}\label{subsec:Proof2}

For this section, we prove the main theorem concerning the solution of \eqref{eqn:TwoPeaksDifferentDirections}, and we recall that we suppress the $\varepsilon$ notation for the solutions to the transformed problems. In order to prove \cref{thm:TwoPeaksMainResult}, we aim to determine local growth rates (in terms of the eigenvalues $\lambda_{i,\varepsilon}$) of solutions.
\begin{lemma}\label{lma:EigenvectorDecay2}
        Let $p_{i,\varepsilon}^{\infty}(x)$ be the $L_\infty(\mathbb{R})$ normalized solution to \eqref{eqn:TwoPeaksDifferentDirectionsLinEigen}. We have the following bounds:
    \[e^{-\frac{\underline{\kappa}_{i}|x-x_\varepsilon|}{\varepsilon}}\leq{}p_{i,\varepsilon}^{\infty}(x)\leq{}\min\left\{1,e^{\frac{-\overline{\kappa}_{i}(|x|-R_0)}{\varepsilon}}\right\},~~x\in\mathbb{R},\]
    where $x_\varepsilon$ is a point such that $p_{i,\varepsilon}^{\infty}(x_\varepsilon)=1$, $\underline{\kappa}_i=\frac{c_i}{2}$ and $\overline{\kappa}_i=\frac{-c_i+\sqrt{2c_i^2+4\delta}}{2}$. 
    Moreover, we have that $\Vert{}p_{i,\varepsilon}\Vert_{L^\infty(\mathbb{R})}\leq{}\frac{K_1}{\varepsilon}$, where $K_1$ is independent of $\varepsilon$.
\end{lemma}
\begin{proof}
        This follows identically to the proof of \cref{lma:DecayEigenvector}. The $L^\infty$ bound on $p_{i,\varepsilon}$ follows from writing  $p_{i,\varepsilon}=\frac{p_{i,\varepsilon}^\infty}{\Vert{p_{i,\varepsilon}^\infty}\Vert_{L^{1}(\mathbb{R})}}\leq{}\frac{2}{\varepsilon{}\underline{\kappa_{i}}}$ where the inequality is a consequence of the lower bound on $p_{i,\varepsilon}^{\infty}$.
\end{proof}

The next lemma is the main ingredient in the proof. We let $W_1(x,t):=N_1(x,t)e^{\int_{0}^{t}\rho(s)ds+\lambda_{1,\varepsilon}{t}}$  which solves
\begin{equation}\label{eqn:22}
\begin{cases}
    \partial_{t}W_1-\mathcal{L}_1W_1=0, \\
    W_1(x,0)=n_0(x),
\end{cases}
\end{equation}
where $\mathcal{L}_1w:=\varepsilon{}c_1\partial_{x}w+\varepsilon^2\partial_{xx}w+w(a_1(x)+a_2(x-\varepsilon{}Ct)-\delta+\lambda_{1,\varepsilon})$. We will now establish that $W_1(x,t)$ decays exponentially in a ball that contains the support of $a_2(x-\varepsilon{}Ct)$. The main result will follow from the bounds we obtain as a corrolarly of this result.

\begin{lemma}\label{lma:LocalisedSuperSoln}
    We assume \ref{assum:B1}-\ref{assum:B4}, and that $a_{1,M}-\frac{c_1^2}{4}>a_{2,M}-\frac{c_2^2}{4}>\delta$. Then $W_1(x,t)\leq{}e^{\frac{K}{\varepsilon}-\left(\eta-\frac{\gamma}{\varepsilon}\right){t}}$ for $(x,t)\in{}[-R_0+\varepsilon{C}t,R_0+\varepsilon{}Ct]\times\left(\frac{\gamma}{\varepsilon},\infty\right)$ where the positive constants $K$ and $\eta$ are independent of $\varepsilon$.
\end{lemma}
\begin{proof}

    To establish the lemma, we construct a supersolution of the form $\overline{W}(x,t)=e^{\varepsilon^{-1}\overline{\phi}(x,t)}$. We now define the function $\overline{\phi}$ and show that $\overline{W}$ satisfies the conditions, given in \cref{subsec:Prelim0}, to be a supersolution.
    Firstly we note that
    \begin{align*}
        \partial_{t}\overline{W}-\mathcal{L}_1\overline{W}&=\overline{W}\left(\partial_{t}\overline{\phi}-\varepsilon{}\partial_{xx}\overline{\phi}-|\partial_{x}\overline{\phi}|^2-c_1\partial_{x}\overline{\phi}-(a_1(x)+a_2(x-\varepsilon{C}t)-\delta+\lambda_{1,\varepsilon})\right)\\
        &=\overline{W}\left(\partial_{t}\overline{\phi}-\varepsilon{}\partial_{xx}\overline{\phi}-\left|\partial_{x}\overline{\phi}+\frac{c_1}{2}\right|^2-\left(a_1(x)+a_2(x-\varepsilon{C}t)-\delta+\lambda_{1,\varepsilon}-\frac{c_1^2}{4}\right)\right)
    \end{align*}

     We will define $\overline{\phi}(x,t)$ piecewise, where each piece is $C^{2}(\mathbb{R}\times(0,\infty))$. Therefore, to establish that it is a supersolution, we require that:
    \begin{enumerate}[label=(C\arabic*)]
        \item  $\overline{\phi}\in{}C^{1}(\mathbb{R},(0,\infty))$.
        \item The differential inequality $\partial_{t}\overline{\phi}-\varepsilon{}\partial_{xx}\overline{\phi}-|\partial_{x}\overline{\phi}|^2-c_1\partial_{x}\overline{\phi}-(a_1(x)+a_2(x-\varepsilon{C}t)-\delta+\lambda_{1,\varepsilon})\geq{0}$ is satisfied in the classical sense wherever $\phi$ is sufficiently smooth.
        \item At each point $(x_0,t_0)$ where $\overline{\phi}$ fails to be $C^{2,1}(\mathbb{R}\times(0,\infty))$ there exists an open neighbourhood $U$ containing $(x_0,t_0)$ and a function $\tilde{w}\in{}C^{2,1}(U)$ such that $\tilde{w}(x,t)\geq{}\phi(x,t)$ for  $(x,t)\in{}U$,  $\tilde{w}(x_0,t_0)=\overline{\phi}(x_0,t_0)$, and $\tilde{w}$ satisfies $\partial_{t}\phi-\varepsilon{}\partial_{xx}\tilde{w}-|\partial_{x}\tilde{w}|^2-c_1\partial_{x}\tilde{w}-(a_1(x)+a_2(x-\varepsilon{C}t)-\delta+\lambda_{1,\varepsilon})\geq{0}$ in the classical sense.
    \end{enumerate}
    
The condition (C3) is satisfied at a boundary point $x_b$ if $\overline{\phi}$ is locally the minimum of two functions which satisfy the differential inequality in the classical sense in a neighborhood of that boundary point. To locally be the minimum of such functions requires that the left derivative is greater than the right derivative, i.e $\partial_{x}\overline{\phi}(x_b+,t)>\partial_{x}\overline{\phi}(x_b-,t)$. To check  that (C3) holds at $x_b$, it is enough to check this condition and there is a small neighbourhood of $x_b$ such that each piece of the $\overline{\phi}$ can be extended smoothly across the boundary, which in particular will be true if the differential inequality $(C2)$ is not strict at $x_b$.

We will now construct $\overline{\phi}$. Firstly, we define the function $\overline{u}_1(x,t)$  as follows
\begin{equation}
    \overline{u}_{1}(x,t)=\begin{cases}
        0,&~~(x,t)\in{}[-R_0,\infty]\times{}(0,\infty),\\
        \frac{k_1}{2}(x+2R_0)^2-\frac{k_1}{2}R_0^2,&~~(x,t)\in{}[-2R_0-L_1,-R_0]\times{}(0,\infty),\\
        -\sqrt{a_{1,M}}(x+2R_0+L_1)-\frac{k_1}{2}R_0^2+\frac{k_1}{2}L_1^2,&~~(x,t)\in{}[-\infty,-2R_0-L_1]\times{}(0,\infty).
    \end{cases}
\end{equation}
Similarly, we define $\overline{u}_2(x,t)$ as follows:
\begin{equation}
    \overline{u}_{2}(x,t)=\begin{cases}
        0,&~~(x,t)\in{}[-\infty,R_0]\times{}(0,\infty),\\
        \frac{k_2}{2}(x-2R_0)^2-\frac{k_2}{2}R_0^2,&~~(x,t)\in{}[R_0,2R_0+L_2]\times{}(0,\infty),\\
        \sqrt{a_{2,M}}(x-2R_0-L_1)-\frac{k_2}{2}R_0^2+\frac{k_2}{2}L_2^2,&~~(x,t)\in{}[2R_0+L_2,\infty]\times{}(0,\infty).\\
    \end{cases}
\end{equation}
We take  \[\overline{\phi}(x,t)=\max\left\{\overline{u}_1(x,t)-\frac{c_1x}{2}+\varepsilon\nu_{1}t,\overline{u}_2(x-\varepsilon{C}t,t)-\frac{c_2(x-\varepsilon{}Ct)}{2}+\varepsilon\nu_2t\right\}+K,\]
where $\nu_1=(\lambda_{1,\varepsilon}-\lambda_{1,0})+\varepsilon{}^{\frac{1}{2}}$, and $\nu_2=\varepsilon{}^{\frac{1}{2}}-\eta$, where $\eta>0$. 

We show that conditions (C1)-(C3) hold for each piece of $\overline{\phi}(x,t)$. The continuity of $\overline{\phi}$ is clear by the construction. We check that (C2) and (C3) hold for $\overline{u}_1-\frac{c_1x}{2}+\nu_1t$. The differential inequality becomes
\[\nu_1-\varepsilon\partial_{xx}\overline{u}_1-|\partial_{x}\overline{u}_1|^2-\left(a_1(x)+a_2(x-\varepsilon{C}t)-\frac{c_1^2}{4}-\delta+\lambda_{1,\varepsilon}\right)\geq{0}.\]
We pick $\gamma=\frac{2R_0}{C}$, so the supports of $a_1(x)$ and $a_2(x-\varepsilon{C}t)$ are separated sufficiently that the $a_2$ term vanishes. A sufficient condition is then
\[\varepsilon^{\frac{1}{2}}-\varepsilon\partial_{xx}\overline{u}_1-|\partial_{x}\overline{u}_1|^2+a_{1,M}-a_1(x)\geq{0}.\]
The inequality is clearly satisfied for the constant part. For it to be satisfied at the quadratic part we require 
\begin{equation*}
    \varepsilon^{\frac{1}{2}}-\varepsilon{}k_1-k_1^2|x+2R_0|^2+a_{1,M}\geq{0},
\end{equation*}
which, since $|x+2R_0|^2$ is largest at the boundaries, amounts to satisfying 
\begin{align}
    \varepsilon^{\frac{1}{2}}-\varepsilon{}k_1-(k_1R_0)^2+a_{1,M}\geq{0}\label{eqn:C2u_11},\\
    \varepsilon^{\frac{1}{2}}-\varepsilon{}k_1-(k_1L_1)^2+a_{1,M}\geq{0}.\label{eqn:C2u_12}
\end{align}
To satisfy (C3) at $x=-R_0$ we require that the left hand derivative $\partial_{x}\overline{u}_1(-R_0+,t)>0$  which is clear for any choice of $k_1>0$. In this case we can take $\tilde{w}(x,t)$ as the quadratic part beyond $-R_0$ since it will still satisfy the differential inequality, and because the inequality is slack in the derivative condition, this implies $\overline{u}_1(x,t)$ is locally a minimum of two functions.  To satisfy (C3) at the boundary $x=-2R_0-L_1$ we similarly require that  \begin{equation}\label{eqn:C3LastBoundary}
    -\sqrt{a_{1,M}}>\partial_{x}\overline{u}_1((-2R_0-L_1)+,t)=-k_1L_1.
\end{equation}
We can satisfy \eqref{eqn:C2u_11}, \eqref{eqn:C2u_12}, and \eqref{eqn:C3LastBoundary} simultaneously by choosing $k_1=\sqrt{\frac{a_{1,M}}{2}}$ and $L_1$ to be any positive number such that $\frac{a_{1,M}}{k_1^2}<L_1^2<\frac{{a_{1,M}+\varepsilon^\frac{1}{2}}-\varepsilon}{k_1^2}$, which is possible provided $\varepsilon<k_1^{-\frac{1}{2}}$.

The approach for $\overline{u}_2(x,t)$ is similar, however, the the differential inequality is now
\[\varepsilon^{1/2}-\eta-\varepsilon\partial_{xx}\overline{u}_2-|\partial_{x}\overline{u}_1|^2-\left(a_2(x)-\frac{c_2^2}{4}-\delta+\lambda_{1,\varepsilon}\right)\geq{0}.\]
Using the fact that $\lambda_{2,0}=-a_{2,M}+\frac{c_2^2}{4}+\delta$, a sufficient condition is then
\[\varepsilon^{1/2}-\varepsilon\partial_{xx}\overline{u}_2-|\partial_{x}\overline{u}_1|^2+(\lambda_{2,0}-\lambda_{1,\varepsilon})-\eta\geq{0},\]
where we note that $(\lambda_{2,0}-\lambda_{1,\varepsilon})>0$ for sufficiently small $\varepsilon$ because $a_{1,M}-\frac{c_1^2}{4}>a_{2,M}-\frac{c_2^2}{4}>\delta$. We can now pick $\eta$ as any number satisfying $0<\eta<(\lambda_{2,0}-\lambda_{1,\varepsilon})$. We may take $\varepsilon$ smaller than $\min\{k_{1}^{-\frac{1}{2}},k_{2}^{-\frac{1}{2}}\}$ which depends only on $a_{1,M}$ and $a_{2,M}$.

It remains to deal with the intersection point $z(t)$ between between $\overline{u}_1(x,t)-\frac{c_1x}{2}+\nu_1t$ and $\overline{u}_2(x-\varepsilon{C}t,t)-\frac{c_2(x-\varepsilon{}Ct)}{2}+\nu_2t$. We note that because we have assumed $c_1<0<c_2$ we easily satisfy (C3) if $z(t)\in[R_0,-R_0+\varepsilon{C}t]$. We will pick $\eta$ such that $z(t)$ is contained in this interval for all time. It is  in this interval initially by the choice $t>\frac{2R_0}{\varepsilon{}C}$, and the formula is given by
\begin{align*}
    z(t)&=\varepsilon\frac{2}{C}\left(v_2-v_1+\frac{Cc_2}{2}\right)t\\
    &=\varepsilon\frac{2}{C}\left(-\eta-\lambda_{1,\varepsilon}+\lambda_{1,0}+\frac{Cc_2}{2}\right)t.
\end{align*}

From this we see the condition that needs to be satisfied is 
\[0<-\eta-\lambda_{1,\varepsilon}+\lambda_{1,0}+\frac{Cc_2}{2}<\frac{C^2}{2},\]
but since $c_2<c_2-c_1=C$ this is satisfied for any $-\lambda_{1,\varepsilon}+\lambda_{1,0}<\eta<-\lambda_{1,\varepsilon}+\lambda_{1,0}+\frac{Cc_2}{2}$. Thus, because $\lambda_{1,\varepsilon}-\lambda_{1,0}>0$  we can pick 
$\eta=\min\left\{\frac{\lambda_{2,0}-\lambda_{1,0}}{2},\frac{Cc_2}{4}\right\}$. 

Lastly, we need to ensure that  ${W}_1\left(x,\frac{\gamma}{\varepsilon}\right)\leq{}e^{\varepsilon^{-1}\overline{\phi}\left(x,\frac{\gamma}{\varepsilon}\right)}$. We observe that, by the comparison theorem, $W_1(x,t)\leq{}e^{C_1-C_2|x|+(d_0-\lambda_{1,\varepsilon})t}$ where we recall $d_0$ from assumption (B1) and $C_1$ and $C_2$ from assumption (B3). Next because $a_{1,M}-\frac{c_1^2}{4}>a_{2,M}-\frac{c_2^2}{4}>\delta$ we have that $\overline{\phi}(x,t)$ is decreasing in $x$ for $x<-2R_0-L_1$ and increasing in $x$ for $x>2R_0+L_2$, hence we can pick a $K>C_1+\frac{2R_0}{C}$ large enough to ensure $\overline{W}\left(x,\frac{2R_0}{\varepsilon{}C}\right)>W_1\left(0,\frac{2R_0}{\varepsilon{}C}\right)$. The conclusion follows by applying the comparison theorem in \cref{subsec:Prelim0}.
\end{proof}

\begin{lemma}\label{lma:ConvergenceofTransform}
Under the same assumptions as in \cref{lma:LocalisedSuperSoln} there is a constant $\alpha_\varepsilon>0$ such that $\left\Vert{}W_1(x,t)-\alpha_\varepsilon{}p_{1,\varepsilon}(x)\right\Vert_{L^{\infty}(\mathbb{R})}\xrightarrow[t\rightarrow\infty]{}0$.
\end{lemma}

\begin{proof}
    We consider $U(x,t)=W_1(x,t+T+\frac{\gamma}{\varepsilon})e^{-\frac{K}{\varepsilon}}$ which solves 
    \begin{equation}\label{eqn:29}
\begin{cases}
    \partial_{t}U-\mathcal{L}U-a_2(x-\varepsilon{}Ct-\gamma{}C-\varepsilon{}T)U=0~~(x,t)\in\mathbb{R}\times\left(0,\infty\right), \\
    U\left(x,0\right)=W_1\left(x+T,\frac{\gamma}{\varepsilon}\right)e^{-\frac{K}{\varepsilon}},~~x\in\mathbb{R},
\end{cases}
\end{equation}
where now $\mathcal{L}u:=-\varepsilon{}c_1\partial_{x}u-\varepsilon^{2}\partial_{xx}u-(a_1(x)+\lambda_{1,\varepsilon})$.
    We note that $0<a_2(x-\varepsilon{}Ct-\gamma{}C-\varepsilon{}T)U\leq{}a_{2,M}e^{-\eta{t}-\eta{}T}$ due to \cref{lma:LocalisedSuperSoln}. Therefore, we have that $\tilde{U}(x,t)\leq{}U(x,t)\leq{}\tilde{U}(x,t)e^{a_{2,M}e^{-\eta{}T}\int_{0}^{t}e^{-\eta{s}}ds}$ for all $(x,t)\in{}\mathbb{R}\times\left(\frac{\gamma}{\varepsilon},\infty\right)$ where $\tilde{U}(x,t)$ solves
        \begin{equation}\label{eqn:30}
\begin{cases}
    \partial_{t}\tilde{U}-\mathcal{L}\tilde{U}=0~~(x,t)\in\mathbb{R}\times\left(0,\infty\right), \\
    U\left(x,0\right)=W_1\left(x,T+\frac{\gamma}{\varepsilon}\right)e^{-\frac{K}{\varepsilon}},~~x\in\mathbb{R},
\end{cases}
\end{equation}
but this is exactly the eigenvalue problem in the case of a single shifting peak, and we therefore know that $\tilde{U}(x,t)=\alpha_{\varepsilon,T}p_{1,\varepsilon}(x,t)+\tilde{\Sigma}(x,t)$ where $\tilde{\Sigma}(x,t)$ decays exponentially. The constant $\alpha_{T,\varepsilon}$ is positive and depends only on the initial condition (or, effectively, $T$) and $\varepsilon$. 

This shows that
\begin{equation}\label{eqn:Tbound}
    \alpha_{\varepsilon,T}p_{1,\varepsilon}(x)-\Sigma_{T}(x,t)\leq{}W_1(x,t)\leq{}e^{a_{2,M}\eta^{-1}e^{-\eta{}T}}\left(\alpha_{\varepsilon,T}p_{1,\varepsilon}(x)+\Sigma_{T}(x,t)\right),
\end{equation}
for $(x,t)\in\mathbb{R}\times{}[T+\frac{\gamma}{\varepsilon},\infty)$.

We first claim that $\alpha_{\varepsilon,T}$ can be bounded above independently of $T$.  To show that it is bounded, we take $T=0$ on the right hand side of \eqref{eqn:Tbound}, and arbitrary $T$ on the left. Sending $t\rightarrow\infty$ then implies that $\alpha_{\varepsilon,T}p_{1,\varepsilon}(x)\leq{}K_1\alpha_{\varepsilon,0}p_{1,\varepsilon}(x)$ where $K_1$ is independent of $\varepsilon$ and $T$. Thus $\alpha_{\varepsilon,T}\leq{}K_1\alpha_{0,\varepsilon}$. This enables us also to find a subsequence $\alpha_{\varepsilon,T_j}$ such that $\alpha_{\varepsilon,T_j}\xrightarrow[j\rightarrow\infty]{}\alpha_{\varepsilon,\infty}$.

We can now argue by contradiction to show the lemma. We suppose there is no constant $\alpha_\varepsilon$ such that $\left\Vert{}W_1(x,t)-\alpha_\varepsilon{}p_{1,\varepsilon}(x)\right\Vert_{L^{\infty}(\mathbb{R})}\xrightarrow[t\rightarrow\infty]{}0$. In particular, there exists a constant $\nu>0$ and sequence of times $t_i$ and points $x_i$ such that $|W_1(x_i,t_i)-\alpha_{\varepsilon,\infty}p_{1,\varepsilon}(x_i)|>\nu$. We can further assume $x_i$ are bounded, since $W_1(x,t)$ decays exponentially (from \cref{lma:EigenvectorDecay2,lma:SuperSoln}) and otherwise we would have a contradiction with the preceeding inequality. We can now pick a convergent subsequence (relabelling) $x_i\xrightarrow[i\rightarrow\infty]{}z$ so that, because $W_1(x,t)$ is locally uniformly bounded, $W_1(x_i,t_i)\xrightarrow[i\rightarrow\infty]{}w_1$ and $|w_1-\alpha_{\varepsilon,\infty}p_{1,\varepsilon}(z)|>\nu$. Without loss of generality, we let $\{T_j\}_{j}$ to be equal to $\{t_i\}_{i}$. Taking $x=x_j$, $t=t_i$, $T=t_i$ and sending $i\rightarrow\infty$ in \eqref{eqn:Tbound} yields $w_1-\alpha_{\varepsilon,\infty}p_{1,\varepsilon}(z)=0$, which contradicts $|w_1-\alpha_{\varepsilon,\infty}p_{1,\varepsilon}(z)|>\nu>0$. 

\end{proof}

We note that we have not shown an exponential decay in time for $\Sigma(x,t)=W_1(x,t)-\alpha_{\varepsilon,\infty}p_{1,\varepsilon}(x)$, unlike the situation in the proof of \cref{lma:ConvergenceToEigenvalue} and in \cite{figueroa2018long,iglesias2021selection}. Nevertheless, we get immediately equivalent results to \cref{lma:ConvergenceToEigenvalue,lma:rhoConvergence} with no modifications to the proof.

\begin{lemma}\label{lma:ConvergenceToEigenvalueTwoPeak}   Under the assumptions $a_{1,M}-\frac{c_1^2}{4}>a_{2,M}-\frac{c_2^2}{4}>\delta$, $\varepsilon$ is small enough, and \ref{assum:B1}-\ref{assum:B4} the normalized population will converge to  $p_\varepsilon$ as $t\rightarrow\infty$, i.e.
         \[\left\Vert{\frac{N_1(x,t)}{\rho_\varepsilon(t)}-{p}_{1}(x)}\right\Vert_{L_{\infty}(\mathbb{R})}\xrightarrow[t\rightarrow\infty]{}0.\]
\end{lemma}
\begin{lemma}\label{lma:rhoConvergenceTwoPeak} Under the assumptions $a_{1,M}-\frac{c_1^2}{4}>a_{2,M}-\frac{c_2^2}{4}>\delta$, $\varepsilon$ is small enough, and \ref{assum:B1}-\ref{assum:B4} the total population $\rho_\varepsilon(t)$ will convergence to a finite, positive value
    \[\rho_\varepsilon(t)\xrightarrow[t\rightarrow\infty]{}\int_{}a{p}_{1}dy.\]
\end{lemma}
\begin{proof}[Proof of \cref{thm:TwoPeaksMainResult}]
 \cref{thm:TwoPeaksMainResult} now follows by combining the preceding two lemmas with  \cref{thm:EigenvalueConvergence} for the case of a single peak. In this case, there is a unique viscosity solution. In the case that $a_{2,M}-\frac{c_2^2}{4}<\delta$ the proof only requires minor modifications.
   
\end{proof}

\section{Numerical computations}\label{sec:NumericalResults}

To complement theoretical results, we present several numerical examples which illuminates interesting features of the transient dynamics, as well as demonstrating our conclusions for the long-time behaviour.

\subsection{Description of methods}

To obtain the numerical results, we make use of a simple finite difference scheme. We discretize space, which we take as $[0,L]$ as $x_{i}=\delta{x}i$ for $i=1,...N_{x}$ where $x_{N_{x}}=L$. We discretize the time interval $[0,T]$ as $t_{i}=\delta{t}i$ for $i=1,...,N_{t}$ where $\delta{t}N_{t}=T$. Iterations are computed using the forward in time Euler scheme with a centre difference approximation of the Laplacian:
\[n^{k+1}_{j}=\begin{cases}
    n^{k}_{j}+\frac{\delta{t}}{\delta{x}^2}(n^{k}_{j+1}-2n^{k}_{j}+n^{k}_{j-1})+\delta{t}n^{k}_{j}a(x_{j},t_{k+1}) &\text{ for } j=2,...,N_{x}-1,k=1,...,N_{t}\\
    0 &\text{ for } j=1,N_{x},k=1,...,N_{t}\\
    
\end{cases}\]
Although we're seeking to approximate a solution on an unbounded domain, we apply Dirichlet boundary conditions rather than, for instance, truncating the spatial domain at each time point, because as the solution is expected to concentrate we suppose that the boundary conditions ultimately do not significantly impact the final solution.

We also make use of an adaptation of the asymptotic preserving scheme given in \cite{calvez2022concentration}. For a detailed description, we refer to their paper but will summarise the key points. This scheme works with the Hamilton-Jacobi equation 
\begin{equation} \label{eqn:RewriteHJB}
\begin{cases*}
\partial_{\tau}\bar{u}_{\varepsilon}+\vert\partial_{x}\bar{u}_\varepsilon-\frac{c}{2}\vert^2=\varepsilon\partial_{xx}\bar{u}_\varepsilon-\left(a(x)-\frac{c^2}{4}-\rho_\varepsilon(t)\right)\\
\rho_\varepsilon(t)=\int{}e^{-\frac{\bar{u}_\varepsilon}{\varepsilon}}dx
    \end{cases*}
\end{equation}
so that $N=e^{\frac{-\bar{u}_\varepsilon}{\varepsilon}}$ solves \eqref{eqn:SameVelocityShifted}.

Where $\bar{u}^{0}$ is specified, the iterations are then given by
\begin{equation}\label{eqn:APSeps}
  \begin{cases*}
\frac{\bar{u}^{n+1}_{i}-\bar{u}^{n}_{i}}{\Delta{t}}+H\left(\frac{\bar{u}^{n}_{i}-\bar{u}^{n}_{i-1}}{\Delta{}x}-\frac{c}{2},\frac{\bar{u}^{n}_{i+1}-\bar{u}^{n}_{i}}{\Delta{}x}-\frac{c}{2}\right)=\varepsilon\frac{\bar{u}^{n}_{i+1}-2\bar{u}^{n}_{i}+\bar{u}^{n}_{i-1}}{\Delta{x}}-\left(a(x_{i})-\frac{c^2}{4}-\rho_{n}\right)\\
\rho_{n}=\Delta{x}\sum_{i\in\mathbb{Z}}e^{-\frac{\bar{u}^{n}_{i}}{\varepsilon}}.
    \end{cases*}  
\end{equation}
Here 
\[H(p,q)=\max{}\{H^{+}(p),H^{-}(q)\}\]
where
\[H^{+}(p)=
\begin{cases*}
    p^2 \text{ if } p>0,\\
    0 \text{ if } p<0.
\end{cases*}\]
and
\[H^{-}(q)=
\begin{cases*}
    0 \text{ if } q>0,\\
    q^2 \text{ if } q<0.
\end{cases*}\]
When $c=0$ this is exactly the scheme $S_\varepsilon$ in \cite{calvez2022concentration}. We note that although the actual schemes are almost identical, we have quite different assumptions on the growth term, which is for us given by $R(x,I)=a(x)-\frac{c^2}{4}-I$. In particular, for $I=0$ this is negative for $|x|>R_{0}$ but in \cite{calvez2022concentration} there exists an $I_{m}>0$ such that $R(x,I_{m})>0$ for all $x$. We find that the scheme converges to what is expected but to prove the convergence is beyond the scope of the current paper and we leave it for future work.

Associated with the above scheme is also a limit scheme 
\begin{equation}\label{eqn:APSLimiting}
\begin{cases*}
\frac{v^{n+1}_{i}-v^{n}_{i}}{\Delta{t}}+H\left(\frac{v^{n}_{i}-v^{n}_{i-1}}{\Delta{}x}-\frac{c}{2},\frac{v^{n}_{i+1}-v^{n}_{i}}{\Delta{}x}-\frac{c}{2}\right)=-\left(a(x_{i})-\frac{c^2}{4}-P_{n}\right)\\
\min_{i\in\mathbb{Z}}v_{i}^{n+1}=0.
\end{cases*}   
\end{equation}
Here $P_{n+1}$ is the limiting value of $\rho_{n+1}$ as $\varepsilon\rightarrow{0}$ which is unique according to \cite{calvez2022concentration}. This is computed by finding the root of the following function
\[J{}\mapsto\min_{i\in\mathbb{Z}}\left\{v^{n+1}_{i}-\Delta{t}H\left(\frac{v^{n}_{i}-v^{n}_{i-1}}{\Delta{}x}-\frac{c}{2},\frac{v^{n}_{i+1}-v^{n}_{i}}{\Delta{}x}-\frac{c}{2}\right)-\Delta{t}R(x_{i},J)\right\}.\]
This function's unique root is $P_{n+1}$ according to Remark 4.3 in \cite{calvez2022concentration}.

\subsection{Results}

The following numerical simulations provide additional support for our conclusions about the long-term behaviour of solutions. Moreover, they also provide insights into the transient behaviour which is not captured by the theorems in \cref{sec:MainResults}.

\begin{figure}[H]
    \centering
    \includegraphics[width=0.7\textwidth]{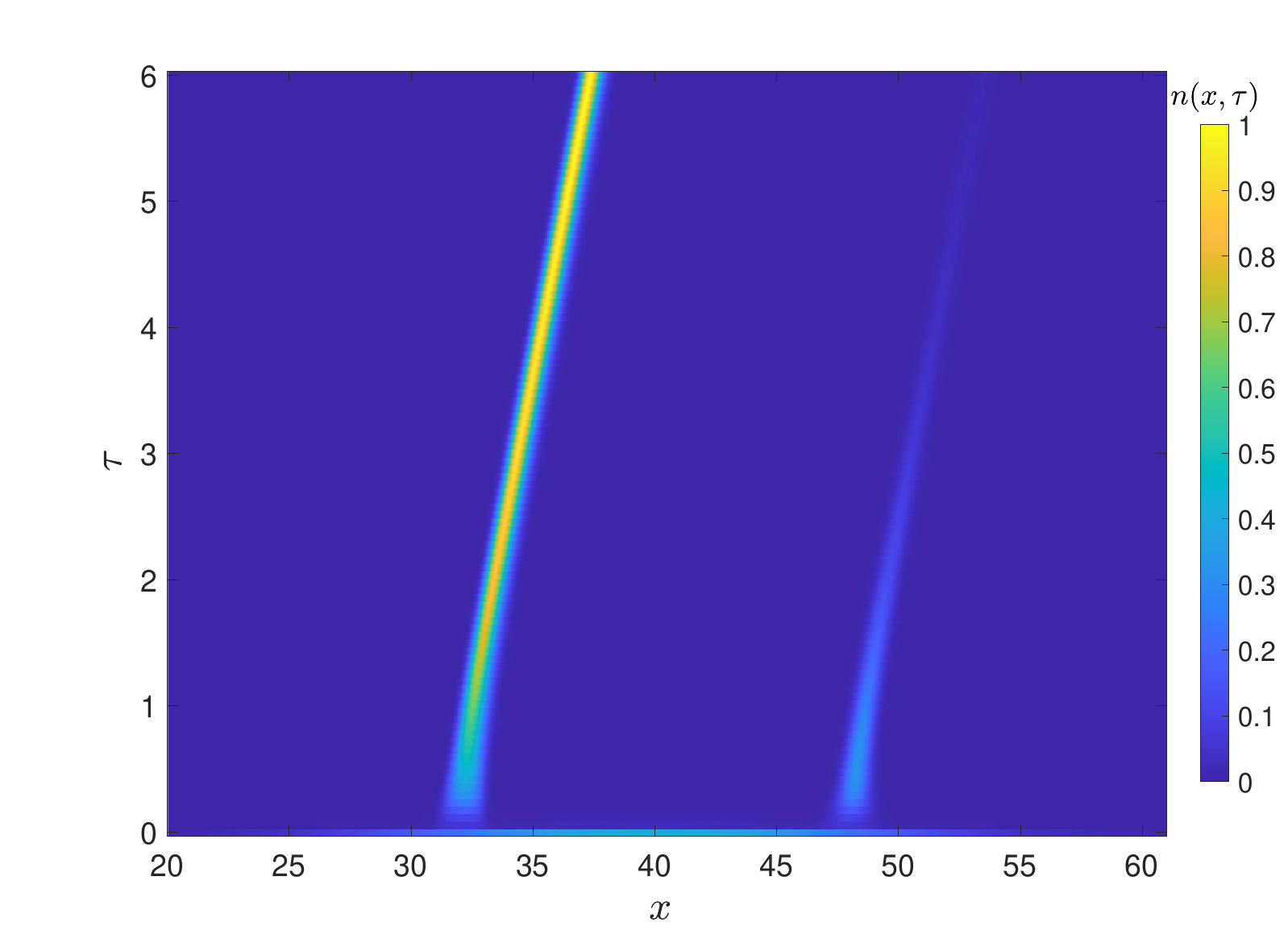}
    \caption{Dynamics of $n(x,\tau)$ on a fixed space and time interval. The initial condition is taken as a $n_{0}(x)=\frac{1}{10}e^{-\frac{(x-37.5)^2}{10^2}}$.  We choose $\delta=-\frac{1}{2}$, $a_{1}(x)=(\frac{5}{2}-(x-35)^4)^{+}$, $a_{2}=\left(\frac{5}{2}-(x-40)^2\right)^{+}$, $\varepsilon=0.1$ and $c_{1}=c_{2}=1$.}
    \label{fig:Theorem1}
\end{figure}

 In \cref{fig:Theorem1} we plot the solution $n(x,\frac{t}{\varepsilon})$. We see that the solution does in fact concentrate, as proved in Theorem 2, on the lagged optimum following the global optimum which satisfies $x_{i}=\text{argmin}_{z\in\{x_{1},x_{2}\}}|a_{i}''(z)|$. This is not captured by \cref{thm:RescaleConvergence} although it is suggested by it.
 
We can compare results using the other numerical scheme too.  In \cref{fig:APS}, we plot $\bar{u}$ which is the numerical solution approximating $u_\varepsilon=-\varepsilon\log{N}$ at four time points. Since the solution is time-dependent, there is a time-dependent minimum $x(T)$ where $u(x_{T},T)=0$. We find that, as expected, $x(T)$ approaches the $\bar{x}_{1}$ both for the $\varepsilon>0$ scheme and the limit scheme.

Moreover, the results show that the solution initially concentrates near both peaks, even though ultimately one dies out. This suggests that both subpopulations (following $x_{1}$ and $x_{2}$) will coexist for some significant time (recall the time units are $\tau=\frac{t}{\varepsilon}$).

We can also compare this to results for an asymptotic scheme which we adapt from \cite{calvez2022concentration} in \cref{fig:APS}. We observe that the limiting scheme and the $\varepsilon>0$ scheme both reproduce the observation of the finite differences method where the solution concentrates only at the lagged optimum $\bar{x}_{i}$ associated to peak with minimal $|a''(x_{i})|$.
\begin{figure}[H]
    \centering
    \includegraphics[width=0.8\textwidth]{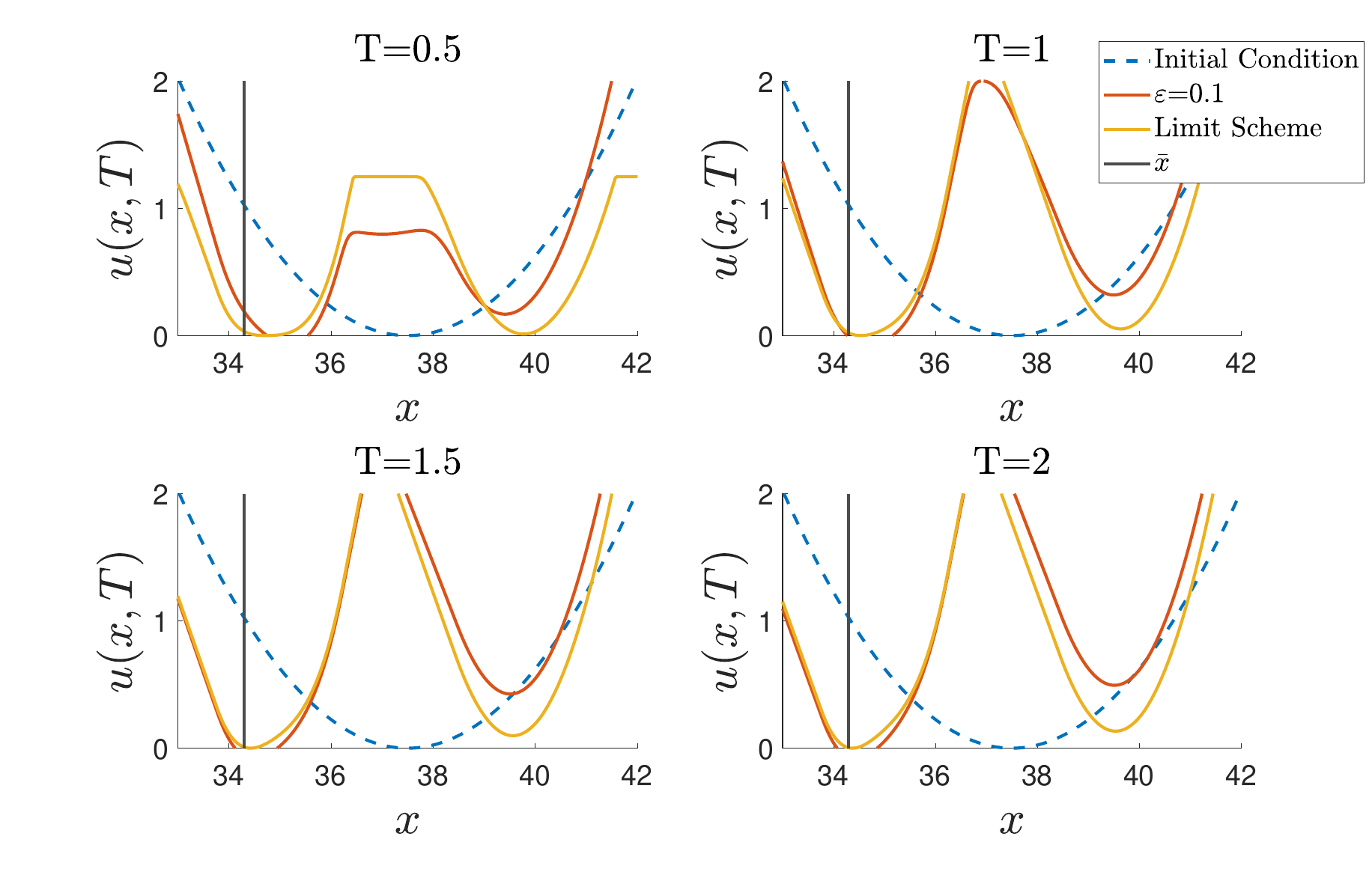}
    \caption{The solution $u(x,T)$ to \eqref{eqn:APSeps} and \eqref{eqn:APSLimiting} for $T\in\{\frac{1}{2},1,\frac{3}{2},2\}$. The initial condition is taken as $u_{\varepsilon}(x,0)=-\varepsilon\left(\log(\frac{1}{10})-\frac{(x-37.5)^2}{10^2}\right)$.  We choose $\delta=-\frac{1}{2}$, $a_{1}(x)=(\frac{5}{2}-(x-35)^4)^{+}$, $a_{2}=\left(\frac{5}{2}-(x-40)^2\right)^{+}$, $\varepsilon=0.1$ and $c_{1}=c_{2}=1$.}
    \label{fig:APS}
\end{figure}

Next, we investigated the effect of increasing the speed $c_{2}$ while leaving $c_{1}$ fixed. As shown in in \cref{fig:SpeedVariation}, when $c_{2}$ is small, $a_{1,M}-\frac{c_{1}^2}{4}<a_{2,M}-\frac{c_{2}^2}{4}$ and so we expect the solution to concentrate on $\bar{x}_{2}+\varepsilon{}c_{2}t$ in the long-term, as required by \cref{thm:TwoPeaksMainResult}. However when $c_{2}$ is large enough we will instead have $a_{1,M}-\frac{c_{1}^2}{4}>\max\left\{\delta,a_{2,M}-\frac{c_{2}^2}{4}\right\}$ and we expect concentration at $\bar{x}_{1}+\varepsilon{}c_{1}t$ in the long term. Indeed, this occurs, and we also see that for $c_{2}=2.5$ initially the subpopulation following $\bar{x}_{2}+\varepsilon{}c_{2}t$ grows and the subpopulation following $\bar{x}_{1}+\varepsilon{}c_{1}t$ decays. We expect that this is due to the fact that there is initially mass near the true optimum at $x_{2}$ and $a_{2,M}>a_{1,M}$. This allows an initially higher growth rate near ${x}_{2}+\varepsilon{}c_{2}t$ for small times. The competition then suppresses the growth everywhere else. However, due to the shift, the population near ${x}_{2}+\varepsilon{}c_{2}t$ cannot be sustained and starts to lag, allowing the population near $\bar{x}_{1}+\varepsilon{}c_{1}t$, which has the higher lagged fitness, to overtake.

\begin{figure}[H]
    \centering
    \includegraphics[width=0.8\textwidth]{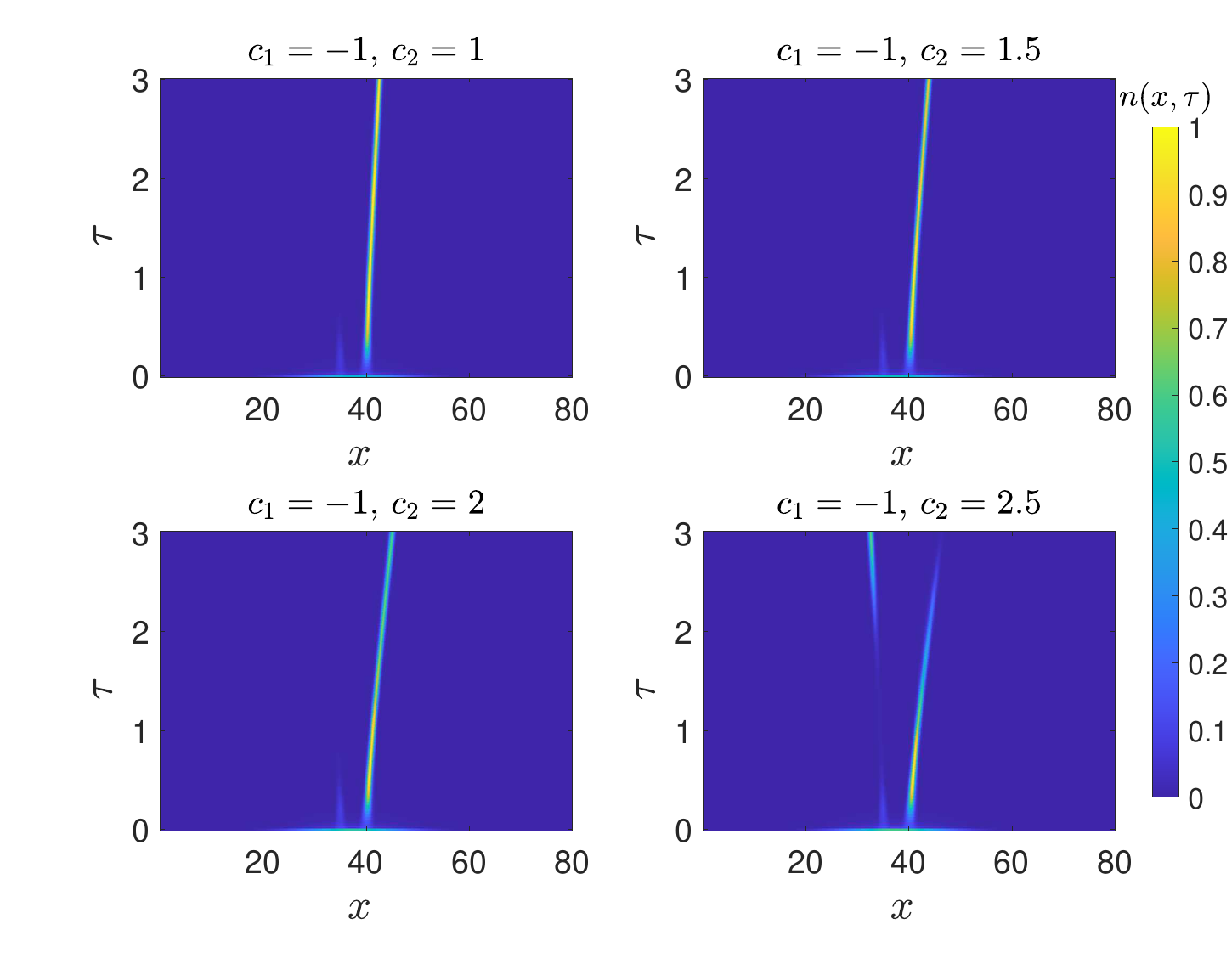}
    \caption{ We choose $\delta=-\frac{1}{2}$, $a_{1}(x)=(\frac{7}{4}-(x-32)^2)^{+}$, $a_{2}=\left(\frac{5}{2}-(x-48)^2\right)^{+}$, $\varepsilon=0.1$ and $c_{1}=-1$ and $c_{2}$ to vary as in the above plots.}
    \label{fig:SpeedVariation}
\end{figure}

We also looked at the effect of initial conditions on the transient behaviour. We keep the parameters fixed and find that the initial conditions can alter where the solution initially concentrates. In particular, for nearby initial conditions, it will concentrate on some point which is likely the single lagged optimum of $a_{1}(x-\varepsilon{}c_{1}t)+a_{2}(x-\varepsilon{}c_{2}t)$ before the sufficiently separate and it instead follows the maximum of the positive lagged optima.

Finally, we determine the behaviour of $\rho_\varepsilon(t)=\int_{\mathbb{R}}n(x,t)dx$, or the total population in \cref{fig:FitnessTrajectory} for an example initial condition where the two peaks overlap. In this case, $\rho_\varepsilon(t)$ is non-linear but eventually monotonic. If this property could be established rigorously, it would simplify some of the proofs presented here, in particular \cref{thm:TwoPeaksMainResult}. 
\begin{figure}[H]
    \centering
    \includegraphics[width=0.8\textwidth]{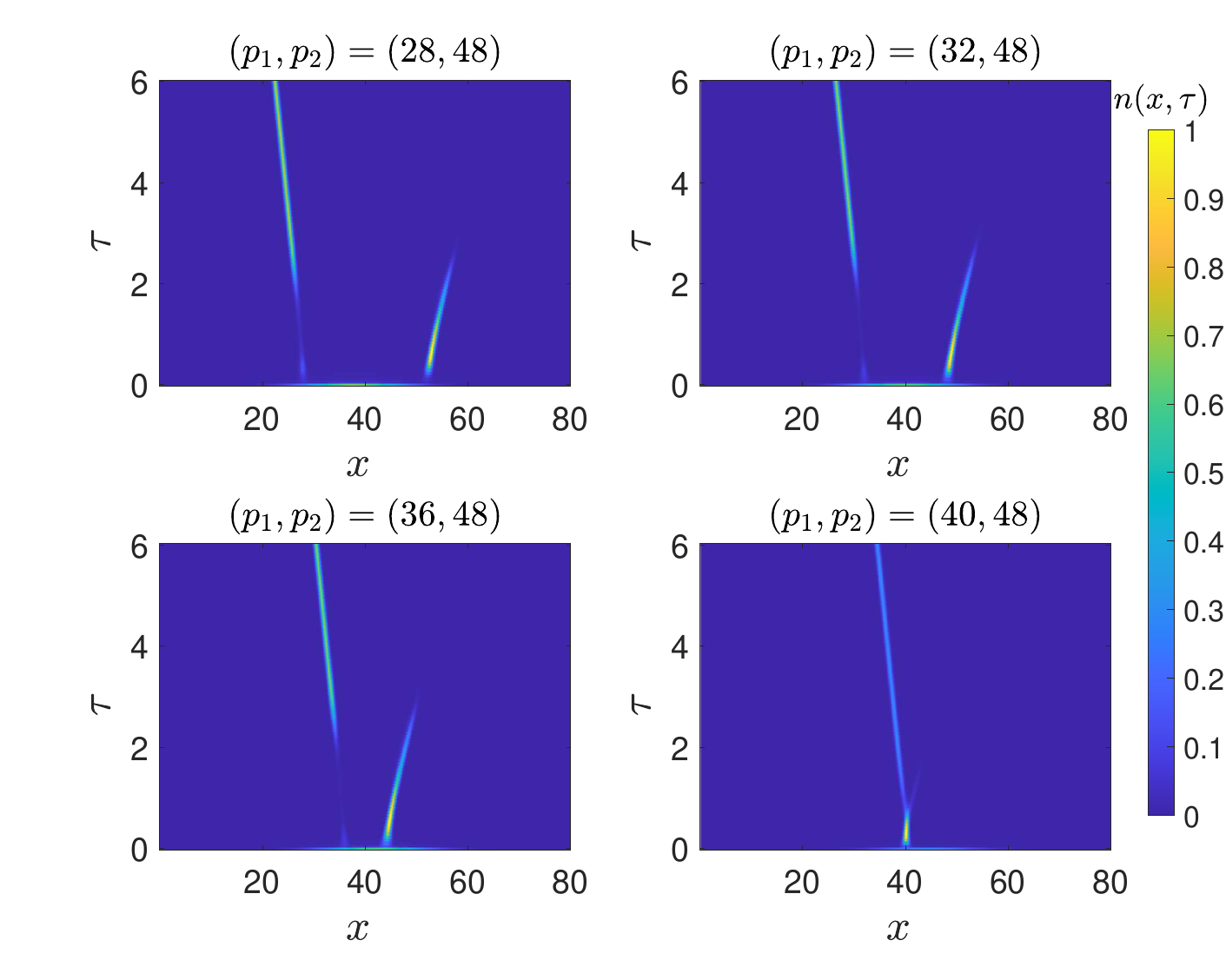}
    \caption{The long-term behaviour of $n(x,\tau)$ for a range of initial conditions. We choose $\delta=-\frac{1}{2}$, $a_{1}(x)=(\frac{7}{4}-(x-p_{1})^2)^{+}$, $a_{2}=\left(\frac{5}{2}-(x-p_{2})^2\right)^{+}$, $\varepsilon=0.1$ and $c_{1}=-1$ and $c_{2}=2.5$. We pick $(p_{1},p_{2})=(28+z,52-z)$ for $z=12,8,4,0$.}
    \label{fig:InitialConditions}
\end{figure}

 We find that for such an initial condition solution concentrates at a point which is initially at neither $\bar{x}_{i}+\varepsilon{}c_{i}t$ due to the overlapping support of $a_{1}$ and $a_{2}$. Once these have separated sufficiently, it concentrates at the lagged optima which has the maximum fitness.

\begin{figure}[H]
    \centering
    \includegraphics[width=0.9\textwidth]{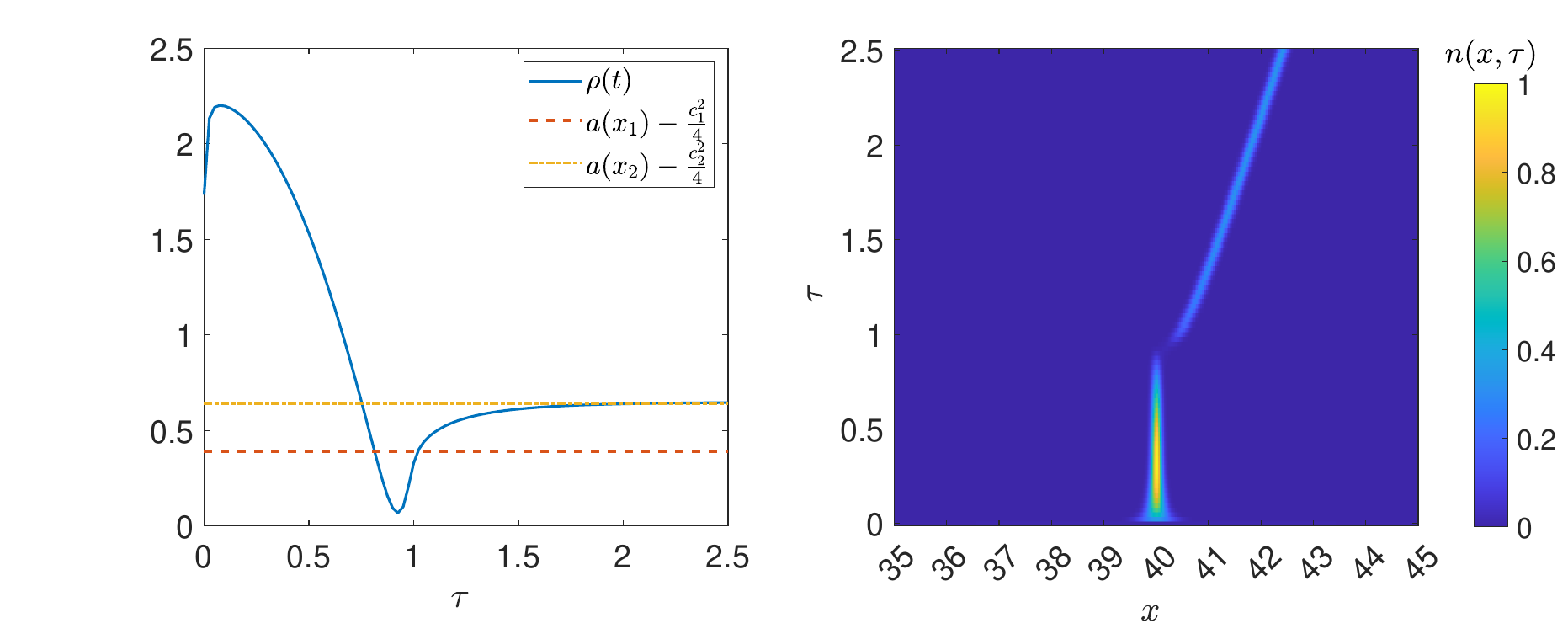}
    \caption{Dynamics of $\rho(\tau)$ and corresponding dynamics of $n(x,\tau)$. We choose $\delta=-\frac{1}{2}$, $a_{1}(x)=(\frac{7}{4}-(x-40)^2)^{+}$, $a_{2}=\left(\frac{5}{2}-(x-40)^2\right)^{+}$, $\varepsilon=0.05$ and $c_{1}=-\frac{6}{5}$ and $c_{2}=\frac{6}{5}$.}
    \label{fig:FitnessTrajectory}
\end{figure}

\section{Discussion and future work}\label{sec:Discussion}

\subsection{Discussion}
Our work has focused on understanding the long-term behaviour of solutions to a novel integro-differential model of asexual reproduction in a temporally changing environment. Our model is novel in that it allows for locally optimal traits to move at different rates. We build on the work in \cite{lorenzi2020asymptotic} which have a static fitness function with multiple global optima and in \cite{iglesias2021selection} which considers a linearly-shifting fitness function. 

Specifically, we have answered the following questions in \cref{thm:EigenvalueConvergence,thm:RescaleConvergence,thm:TwoPeaksMainResult}:
\begin{enumerate}
    \item For which conditions on the shifting rate and intrinsic growth rates does the entire population go extinct?
    \item Which trait/s dominate in long time, if the population does not go extinct?
\end{enumerate}

Theorem 1, which applies to Case 1 where there are multiple global optima shifting with the same speed, shows that the solution concentrates on a subset of the "lagged optima" in the small-mutation and long-time limit. Theorem 2 we prove a weighted rescalling will converge on only the shallowest optima. This, combined with the numerical results, suggest that it concentrates on the lagged optima behind the shallowest peaks. In Case 2, where there are multiple optima shifting at different speeds, Theorem 3 shows that if $a_{1,M}-\frac{c_{1}^2}{4}>\max\{a_{2,M}-\frac{c_{2}^2}{4},\delta\}$ then the solution concentrates on $\bar{x}_{1}$, and if the lagged fitness associated to both $a_{i}$ are negative, then the population goes extinct. In other words, if it does not go extinct, the solution concentrates on lagged optima with largest fitness.

This shows that the dominant subpopulation depends on the lagged optimum fitness and not just the true optimal fitness. From the numerical results in \cref{sec:NumericalResults}, we observe that it can first appear that one subpopulation (the one following the optimum with the true optimal fitness) is dominant only to later be overtaken by the subpopulation with the higher lagged optimum. This is particularly relevant when the population initially begins with the optimal traits overlapping. This would be an interesting feature to look out for in cell populations in an aging environment where there are known to be several evolutionary strategies to cope with the environmental change, i.e the case of decoy fitness peaks mentioned in the introduction. Such a mechanism may be responsible for a sudden emergence of cancer cells, not because they are truly more suited to the environment, but because they can adapt more easily to a changing environment.

\subsection{Limitations and Future Work}
In the specification of this model, we have made some simplifying choices that allow for clearer presentation without sacrificing the generality of the results. In Case 2, by taking $a_{i}$ with compact support and constant $\delta$ we are considering only models which have a \textit{constant} negative fitness away from the optimal traits, whereas the previous model only required that the fitness be bounded above by a negative constant. Ultimately, this should not affect the main conclusions since the solutions localise to the moving lagged optima in the limits we will consider.  Another simplification is that we choose to deal with only two $a_{i}$, each with a unique maximum, but the theory here is applicable to the case of any finite number of $a_{i}$, each with a finite number of maxima. Again, this is because the peaks all separate eventually so only interact through the competition term. 

It would be interesting to extend the results here to higher dimensional trait space which may be more relevant when there is competition between more than two traits. It seems plausible that one could determine which peak has the greatest local growth rate and thus reduce the problem to the relevant shifting one peak problem. In the higher dimensional case, however, we no longer have the explicit representation formula for the solutions to the limiting Hamilton-Jacobi Bellman equation, and further work would be required to show the concentration of the solution.

\appendix
\section{Summary of previous results}\label{sec:Preliminaryresults}

\subsection{Related models}

The problem where $a(x,t)$ is time independent and has multiple peaks in a bounded domain with Neumann boundary conditions is studied in \cite{lorenzi2020asymptotic}. This problem is given by
\begin{equation}\label{eqn:NoShift}
\begin{cases*}
        \partial_{t}n-\sigma\partial_{xx}n=n\left(a(x)-\int_{\mathbb{R}}n(y,t)dy\right), & $(x,t)\in{\Omega}\times{}\mathbb{R}^{+}$, \\
        n(x,0)=n_{0}(x), & \\
        \nabla{n}\cdot\nu(x)=0, &  $(x,t)\in{\partial{\Omega}}\times{}\mathbb{R}^{+}$,\
\end{cases*}  
\end{equation}
where $\nu(x)$ is the outward unit normal. For problems in smooth bounded domains and periodic parabolic problems (in this case, even in unbounded domains) one can use the Krein-Rutnman theorem to assert the existence of a principle eigenvalue and associated positive eigenfunction solution for the linearised problem.  The Krein-Rutman theorem is a generalisation of the well-known Perron–Frobenius theorem (which asserts the existence of a positive eigenvalue and eigenvector for square matrices where all entries are positive) to positive operators on Banach spaces \cite{lam2022introduction}. This is useful because one can usually relate the solutions for arbitrary initial conditions to the solution to the eigenvalue problem which is easier to analyse.

The problem without periodic coefficients and in an unbounded domain excludes the use of the Krein-Rutman theorem. According to \cite{berestycki2015generalizations}: "The Krein-Rutman theory cannot be applied if $\Omega$ is nonsmooth or unbounded (except for {problems in periodic settings}), because the resolvent of $L$ [the operator in question] is not compact." Our problem \eqref{eqn:UnboundedShifting} must be considered in the unbounded domain due to the term $a(x-\tilde{c}t)$. By making a coordinate change to $z=x-\tilde{c}t$ and performing a Liouville transform we can reduce $\eqref{eqn:UnboundedShifting}$ to almost exactly the problem \eqref{eqn:NoShift} except for the fact the domain is unbounded and we need to be able to carrying over results from the transformed problems to the original. In particular we cannot simply undo the Liouville transform once we find the limiting solution as $\varepsilon\rightarrow{0}$ since it depends on $\varepsilon$.

It is shown that the solution to the problem on a bounded domain \eqref{eqn:NoShift} with $\sigma=\varepsilon$ converges as $t\rightarrow\infty$ to a multiple of the solution to the eigenvalue problem
\begin{equation}\label{eqn:NoShiftEigen}
\begin{cases*}
        -\varepsilon\partial_{xx}\psi_\varepsilon-a(x)\psi_\varepsilon=\lambda_\varepsilon\psi_\varepsilon & $x\in{\Omega},$ \\
        \nabla{\psi}\cdot\nu(x)=0 & $x\in {\partial{\Omega}},$\\
        \psi\geq{0} & $x\in\Omega.$
\end{cases*}  
\end{equation}
Furthermore, $\psi_\varepsilon\xrightharpoonup[\varepsilon\rightarrow{0}]{}\sum_{i}a_{i}\delta_{x_{i}}$. From these two facts, one can conclude the solution concentrates on some subset of the maximum points of $a(x)$.  Due to this concentration, it can be expected that the fact that the domain is unbounded should not matter (i.e the highly concentrated solutions should be relatively unaffected by the boundary conditions at any fixed distance as $\varepsilon\rightarrow{0}$).

To establish the existence of a solution $\psi_\varepsilon$ to \eqref{eqn:NoShiftEigen}, and also its convergence to a sum of Dirac deltas, one makes use of Krein-Rutman theorem which gives a variational formula of the eigenvalue (also known as the Rayleigh-Quotient formula).

\begin{lemma}[Lemma 1 in \cite{lorenzi2020asymptotic}]
    There exists an eigenvector, eigenvalue pair $(\psi_\varepsilon,\lambda_\varepsilon)$
 solving \eqref{eqn:NoShiftEigen}. The eigenfunction is unique up to normalization and the eigenvalue is characterised by the following formula:
    \[\lambda_\varepsilon=\inf_{\phi\in{}H^{1}(\Omega)\backslash\{0\}}\frac{\varepsilon\int_{\Omega}|\nabla\phi|^2-\int_{\Omega}a(x)\phi^2}{\int_{\Omega}\phi^2}.\]
\end{lemma}

Using this lemma, one gets that $\psi_\varepsilon$ concentrates by: firstly, choosing an appropriate sequence $\psi_\varepsilon$ such that $\lambda_\varepsilon\rightarrow{-a_{M}}$. Then a simple computation shows, for a test function $\eta\in{C(\overline{\Omega})}$ supported away from $\{x_{1},...x_{n}\}$, that
\[\int_{\Omega}\eta\psi_\varepsilon\xrightarrow[\varepsilon\rightarrow{0}]{}0.\]
Unfortunately, the use of the Rayleigh quotient is lost for unbounded domains except under some specific conditions on the fitness function $a(x)$, for instance, if the fitness function is confining ($\lim_{|x|\rightarrow\infty}a(x)=-\infty$). The condition comes from the study of confining potentials in quantum mechanics, and the authors of \cite{alfaro2019evolutionary} obtain some results regarding evolutionary branching under the assumption the fitness function is a confining potential.

After establishing the basic concentration results, the authors of \cite{lorenzi2020asymptotic} are able to find further constraints for the subset of $\{x_{1},...,x_{n}\}$ where the solution eventually concentrates: their Proposition 2 shows an example of a symmetric fitness function $a(x)$ which leads to equal concentration on two peaks, and their Proposition 3 refines the concentration set according the concavity of $a(x)$ at each of its maxima. They borrow this result from semi-classical analysis, as given in \cite{holcman2006singular} for operators defined on a compact Riemannian manifold, independent of boundary conditions. We will usually consider Dirichlet boundary conditions when working with approximate problems on bounded domains, of the general form:
\begin{equation}\label{eqn:NoShiftEigenDir}
\begin{cases*}
        -\varepsilon\partial_{xx}\psi_\varepsilon-a(x)\psi_\varepsilon=\lambda_\varepsilon\psi_\varepsilon &$x\in{\Omega},$ \\
        {\psi}=0 & $x\in{\partial{\Omega}},$\\
        \psi\geq{0} & $x\in\Omega.$
\end{cases*}  
\end{equation}
We borrow the same result, phrased suitably, for our problem.

\begin{lemma}[Proposition 3 in \cite{lorenzi2020asymptotic}]\label{lma:TopologicalPressure}
    Let $S(x)=|a''(x)|$ for $x\in{}M:=\{x_{1},...,x_{n}\}=\text{argmax}_{x\in\Omega}a(x)$. Let $M_{1}=\text{argmin}_{x_{j}}S(x_{j})$. Then the solution $\psi_\varepsilon$ to \eqref{eqn:NoShiftEigenDir}, satisfies, up to extraction of subsequences:
    \[\psi_\varepsilon\xrightharpoonup[\varepsilon\rightarrow{0}]{}\sum_{x_{i}\in{M}\cap{M_{1}}}a_{i}\delta_{x_{i}}.\]
\end{lemma}

Although in \cite{lorenzi2020asymptotic} this result holds for a Neumann problem, the semiclassical analysis result they use remains applicable to Dirichlet problems.

We will analyse and find the concentration results for the problem \eqref{eqn:UnboundedShifting} under the scaling $\sigma=\varepsilon^2$ and $\tilde{c}=c\varepsilon$. We aim to do this by similarly analysing the eigenvalue problem, using the results from \cite{huska2008exponential} which construct solutions to the eigenvalue problem for such parabolic problems on unbounded domains as a limit of Dirichlet problems on bounded domains.

The authors of \cite{iglesias2021selection} (who take $a(x,t)=a(e(t),x-\tilde{c}t)$ for some periodic function $e(t)$ with period $T$) also make use of theory presented in  \cite{huska2008exponential}.  Under the scaling $\tilde{c}=c\varepsilon$ and $\sigma=\varepsilon^2$ they show that the shifted solution $N_\varepsilon(x,t):=n_{\varepsilon}(x+c\varepsilon{}t,t)$ will concentrate at a point $\bar{x}$ which they call the lagged optima. Letting $\bar{a}(x)=\frac{1}{T}\int_{0}^{T}a(e(t),x)$ and $x_{m}$ be the unique maxima of $\bar{a}$, the authors show that the lagged optima satisfies the following equation
\[\bar{a}(\bar{x})=a(x_{m})-\frac{\tilde{c}^2}{4\sigma}.\]
In particular, if the right-hand side is negative, the population dies out.

We are able to show that the solution $M_{\varepsilon}(x,t)$ to \eqref{eqn:WithoutDrift} concentrates on some set of finite points as $\varepsilon\rightarrow{0}$ by using theory presented \cite{huska2008exponential} to reduce the problem to one in a finite domain, then combining the results in \cite{lorenzi2020asymptotic} and \cite{iglesias2021selection} to determine these locations more precisely, and to show under what conditions we can obtain concentration to a single location. 

It remains is beyond the scope this paper to translate this result to what it means for $n_\varepsilon(x,t)$ the solution to \eqref{eqn:UnboundedShifting}. We predict that the Liouville transform will have simply shifted these concentration points but a precise characterisation of $\varepsilon\log{}(n_\varepsilon)$ would be required to find the locations. This is exactly what one studies when performing a WKB-transform, but for this particular problem difficulties are encountered in determining uniqueness of this solution.

\subsection{Generalised super- and sub- solutions}\label{subsec:Prelim0}

It is necessary to glue together and construct potentially non-smooth super- and sub- solutions throughout this paper. To this end, we review the notion of generalised sub- and super- solutions as discussed in \cite{lam2022introduction} Chapter 1, Section 1, and we paraphrase the Definition 1.1.1 here. Let $\Omega\subset\mathbb{R}^N$ be a possibly unbounded domain, define $\Omega_T=\Omega\times{}(0,T]$. Consider the following  operator (using index notation):
\[\mathcal{L}u:=-a^{ij}\partial_{ij}u-b^{i}\partial_{i}u-cu.\]
We suppose that $a^{ij},b^{i},c\in{}C^{0}(\overline{\Omega}_T)$, and that $a^{ij}$ is uniformly  elliptic, that is there is a $\lambda_0>0$ such that $\lambda_0{}|\zeta|^2\leq{}a^{ij}(x,t)\zeta_i\zeta_j$ for all $\zeta\in\mathbb{R}^N$ and $(x,t)\in\Omega_T$.

\begin{defn}
    A function $\underline{u}\in{}C(\Omega_{T})$ satisfies the inequality 
    \[\partial_{t}\underline{u}-\mathcal{L}\underline{u}\leq{}f(x,t,\underline{u},D\underline{u}),\]
    in the generalised sense if for every $(x_0,t_0)\in\Omega_T$ there exists a neighbourhood $U$ of $(x_0,t_0)$ and a function $\tilde{u}\in{}C^{2,1}(\overline{U})$ such that $\tilde{u}\leq{}u$ in $U$, $\tilde{u}\leq{}\underline{u}$ in $U$, $\underline{u}(x_0,t_0)=\tilde{u}(x_0,t_0)$, and
        \[\partial_{t}\tilde{u}(x_0,t_0)-\mathcal{L}\tilde{u}(x_0,t_0)\leq{}f(x_0,t_0,\tilde{u}(x_0,t_0),D\tilde{u}(x_0,t_0)).\]
        Then $\underline{u}$ is called a generalised subsolution.
\end{defn}
One obtains the following comparison theorem, adapted to unbounded domains (\cite{lam2022introduction}, Chapter 6, Section 6.2, Theorem 6.2.1)
\begin{theorem}\label{thm:UnboundedMaxP}
    Suppose $u$ is a generalised subsolution of 
        \[\begin{cases*}
            \partial_{t}u-\mathcal{L}u\leq{}0~~(x,t)\in\Omega\times{}(0,T),\\
            u(x,t)\leq{}0~~(x,t)\in\partial\Omega\times{}(0,T),\\
            u(x,0)\leq{}0,
        \end{cases*}\]
        which satisfies
        \[\liminf_{R\rightarrow\infty}e^{-kR^2}\left[\max{}_{\substack{x\in\Omega,|x|=R,\\0\leq{}t\leq{T}}}u(x,t)\right]\leq{}0,\]
        then $u\leq{}0$ in $\Omega\times(0,T)$.
\end{theorem}
We will make frequent use of this direct consequence
\begin{theorem}\label{thm:UnboundedComparison}
    Suppose $u$ is a  solution of 
        \[\begin{cases*}
            \partial_{t}u-\mathcal{L}u=0~~(x,t)\in\mathbb{R}\times{}(0,T),\\
            u(x,0)=u_0(x),
        \end{cases*}\]
        and that $v$ is a generalised subsolution \[\begin{cases*}
            \partial_{t}v-\mathcal{L}v\leq{}0~~(x,t)\in\mathbb{R}\times{}(0,T),\\
            u(x,0)=v_0(x).
        \end{cases*}\]
       If $v_0(x)-u_0(x)\leq{0}$ and $w=v(x,t)-u(x,t)$  satisfies
        \[\liminf_{R\rightarrow\infty}e^{-kR^2}\left[\max{}_{\substack{x\in\Omega,|x|=R,\\0\leq{}t\leq{T}}}w(x,t)\right]\leq{}0,\]
        then $v(x,t)\leq{}u(x,t)$ for $(x,t)\in\mathbb{R}\times(0,T)$.
\end{theorem}
The analogous theorems for generalised supersolutions are obtained by flipping the inequalities. 

\subsection{Principle Floquet Bundles for Linear Parabolic Equations with Time Dependent Coefficients}\label{subsec:Prelim}
This reviews the main result we need from \cite{huska2008exponential} and is included in the introduction for completeness. 

Consider the following two problems. First, on the whole space $\mathbb{R}^N$

\begin{equation}\label{eqn:LinParabolic}
    \partial_{t}u-\Delta{u}=A(x,t)u \text{ on } \mathbb{R}^N\times{(s,\infty)}
\end{equation}

where $s$ is some arbitrary number in $\mathbb{R}$. Secondly on a ball of radius $R$, the Dirichlet problem: 

\begin{equation}\label{eqn:LinParabolicDirichlet}
\begin{cases*}
     \partial_{t}u-\Delta{u}=A(x,t)u &$(x,t)\in{}B_{R}\times{(s,\infty)}$,\\
     u\geq{0}, &\\
     u=0 &$(x,t)\in{}\partial{}B_{R}\times{(s,\infty)}$
\end{cases*}
\end{equation}

We take the following assumptions from \cite{huska2008exponential}. Note that all constants are positive.
\begin{enumerate}[start=1,label={(C\arabic*)}]
    \item There are constants $A_{0}$ and $r_{0}$ such that $\Vert{}A\Vert_{L^{\infty}(\mathbb{R}^N\times{\mathbb{R}})}\leq{}A_{0}$, and $A(x,t)\leq{0}$ a.e for $|x|\geq{r_{0}}$.\label{assum:C1}
    \item For each $s=s_{0}$ there is a solution $\phi$ of \eqref{eqn:LinParabolic} such that $\phi(.,t)\in{L^\infty(\mathbb{R}^N)}$ for all $t\geq{}s_{0}$ and for some positive constants $\varepsilon$ and $C$, we have:
    \[\frac{\Vert\phi(.,t)\Vert_{L^\infty(\mathbb{R}^N)}}{\Vert\phi(.,s)\Vert_{L^\infty(\mathbb{R}^N)}}\geq{}Ce^{\varepsilon(t-s)} \text{ for } s_{0}\leq{s}\leq{t}.\]
\end{enumerate}
As is proved in \cite{huska2008exponential} this is equivalent to the hypothesis on \eqref{eqn:LinParabolicDirichlet}, which is  that 
\begin{enumerate}[start=3,label={(C{\arabic*})}]
\item There are constants $R_{0}$, $C_{0}$ and $\varepsilon_{0}$ such that for each $s=s_{0}$ the problem \eqref{eqn:LinParabolicDirichlet} with $R=R_{0}$ has a positive solution $u(.,s_{0})\in{}L^\infty(\mathbb{R}^N)$ and 
\[\frac{\Vert{u}(.,t)\Vert_{L^\infty(\mathbb{R}^N)}}{\Vert{u}(.,s)\Vert_{L^\infty(\mathbb{R}^N)}}\geq{}C_{0}e^{\varepsilon_{0}(t-s)} \text{ for } s_{0}\leq{s}\leq{t}.\]\label{assum:C3}
\end{enumerate}

The latter will be easier to show in general, and as remarked in \cite{huska2008exponential}, the second hypothesis holds for all $R>R_{0}$ if it is shown to hold for $R=R_{0}$.

To state the first theorem we will later use, we need to also introduce the adjoint problem to \eqref{eqn:LinParabolic}
\begin{equation}\label{eqn:LinParabolicAdjoint}
    -\partial_{t}v-\Delta{}v=A(x,t)v \text{ in } \mathbb{R}^N\times{(s,\infty)}.
\end{equation}
This is obtained from \eqref{eqn:LinParabolicDirichlet} by the change of variables $t\rightarrow{-t}$.

We have the following theorem:

\begin{theorem}[From Theorems 2.1 and 2.2 in \cite{huska2008exponential}]
There exist positive solutions $\phi$ of \eqref{eqn:LinParabolic} and $\psi$ of the adjoint problem \eqref{eqn:LinParabolicAdjoint}, both with $s=-\infty$. 

Let\[X_{1}(t)=\text{span}\{\phi(.,t)\}\] and 
\[X_{2}(t)=\{v\in{}L^\infty(\mathbb{R}^N):\int_{\mathbb{R}^N}\psi(x,t){v(x)}dx=0\}.\]
The following are true

\begin{enumerate}[label=(\roman*)]
    \item $X_{1}(t)\oplus{}X_{2}(t)=L^{\infty}(\mathbb{R}^N)$ for all $t\in\mathbb{R}$.
    \item $X_{1}$ and $X_{2}$ are invariant. That is, if $u(.,t;s,u_{0})$ is a solution to \eqref{eqn:LinParabolic} with initial condition $u_{0}\in{}X_{i}(s)$ then $u(.,r;s,u_{0})\in{}X_{i}(t)$ for all $t\geq{s}$.
    \item There are positive constants $C$ and $\gamma$ such that for any $u_0\in{}X_{2}(s)$, we have \[\frac{\Vert{u(.,t;s,u_{0})}\Vert_{L^\infty(\mathbb{R}^N)}}{\Vert{\phi(.,t)}\Vert_{L^\infty(\mathbb{R}^N)}}\leq{}Ce^{-\gamma{}(t-s)}\frac{\Vert{u_{0}}\Vert_{L^\infty(\mathbb{R}^N)}}{\Vert{\phi(.,s)}\Vert_{L^\infty(\mathbb{R}^N)}}~~(t\geq{s})~~.\]
    \end{enumerate}
\end{theorem}

We also state a second lemma which is proved in \cite{huska2008exponential} during the course of proving the above theorem, but not presented as an independent result.

\begin{lemma}\label{lma:LocUniEigenvalueConvergence}
        A subsequence of solutions $\phi_{R_{n}}$ \eqref{eqn:LinParabolicDirichlet} will converge (as $n\rightarrow\infty$) locally uniformly in $\mathbb{R}^N\times[-t,t]$ for all $t>0$ to a solution $\phi$ of \eqref{eqn:LinParabolic}. Hence a positive entire solution to \eqref{eqn:LinParabolic} exists.
\end{lemma}

We will mainly be interested in the following corollary of this theorem, which provides the long-term behaviour for any (sensible) initial condition:

\begin{cor}\label{cor:FloquetBundle}

Given any initial condition $u_{0}\in{}L^{\infty}(\mathbb{R}^N),$ there exists a constant $\alpha$ such that 
\[\frac{\Vert{}u(.,t;s,u_{0})-\alpha\phi(.,t)\Vert_{L^{\infty}(\mathbb{R}^N)}}{\Vert{}\phi(.,t)\Vert_{L^{\infty}(\mathbb{R}^N)}}\xrightarrow[t\rightarrow\infty]{}0.\]
\end{cor}

This is a consequence of noting that we can decompose the initial condition as $u_{0}=\alpha\phi+v$ where $v\in{X_{2}(s)}$. If $u_0\notin{}X_2(0)$, for instance if $u_0\geq{0}$, then $\alpha$ is also non-zero. If $u(x,t)\geq{0}$ for all $t>0$ then $\alpha>0$.

\section{Proofs of technical lemmas}\label{Appendix}
\subsection{Proof of \cref{lma:ViscocitySoln}}\label{App:ProofofLemmaHJ}
Here we prove \cref{lma:ViscocitySoln} by providing some uniform bounds on $\psi_\varepsilon$.  For the analogous eigenvalue problem in \cite{iglesias2021selection} it is remarked that this can be done but since the computations are similar they do not provide a proof, hence we do for completeness. 

\begin{proof}

We recall that we may write 
\[p_\varepsilon(x)=e^{\frac{\psi_\varepsilon(x)}{\varepsilon}},\]
so that $\psi_\varepsilon$ then solves: 
\begin{equation}
        -\varepsilon|\partial_{xx}\psi_\varepsilon|-\left|\partial_{x}\psi_\varepsilon-\frac{c}{2}\right|^2-a(x)+\frac{c^2}{4}-\lambda_\varepsilon=0.
\end{equation}
Since $p_\varepsilon=A_\varepsilon{}p_{\varepsilon}^{\infty}$ where $A_\varepsilon=\Vert{p_\varepsilon^{\infty}}\Vert_{L^{1}(\mathbb{R})}^{-1}$, the bounds from \cref{lma:DecayEigenvector} imply
\[\varepsilon\log(A_\varepsilon)-\underline{\kappa}|x-x_\varepsilon|\leq{}\psi_\varepsilon\leq{}-\overline{\kappa}(|x|-R_0)+\varepsilon\log(A_\varepsilon).\]
Noting that the upper bound in \cref{lma:DecayEigenvector} requires that $|x_\varepsilon|<R_0$, and that $A_\varepsilon<K$ for some constant $K>0$, we have
\begin{equation}\label{eqn:PsiUniBound}
    -K_2-\underline{\kappa}|x|<\psi_\varepsilon<-\overline{\kappa}|x|+K_2
\end{equation}
This shows the locally uniform bounds.

We also need Lipschitz bounds. The Bernstein-type method presented in \cite{iglesias2021selection,figueroa2018long} can be used to obtain these. The idea is to show $\partial_{x}w_\varepsilon$ is a subsolution to some other elliptic PDE to which we can apply a maximum principle and obtain a uniform upper bound.  First we define $w_\varepsilon=\sqrt{2K_2-\psi_\varepsilon}$ which solves:
\[-\varepsilon\partial_{xx}w_\varepsilon-\left(\frac{\varepsilon}{w_\varepsilon}-2w_\varepsilon\right)|\partial_{x}w_\varepsilon|^2-c\partial_{x}w_\varepsilon=\frac{a(x)+\lambda_\varepsilon}{-2w_\varepsilon}.\]
Denote $W_{\varepsilon}=\partial_{x}w_\varepsilon$. We differentiate the above with respect to $x$, and multiply by $\frac{w_\varepsilon}{|w_\varepsilon|}$ to obtain:
\[-c\partial_{x}\vert{}W_\varepsilon\vert-\varepsilon\partial_{xx}\vert{}W_\varepsilon\vert-2\left(\frac{\varepsilon}{w_\varepsilon}-2w_\varepsilon\right)\partial_{x}\vert{}W_\varepsilon\vert{}W_\varepsilon+\left(\frac{\varepsilon}{w_\varepsilon^2}+2\right)\vert{}W_\varepsilon\vert^3=\frac{a'(x)W_\varepsilon}{-2w_\varepsilon{}|W_\varepsilon|}+\frac{\vert{}W_\varepsilon\vert(a(x)+\lambda_\varepsilon)}{2w_\varepsilon^2}.\]
Firstly, using the bounds for $a(x),a'(x)$ and the following bound
\[\sqrt{K_2}\leq{}w_\varepsilon\leq{}\sqrt{K_3|x|},\]
where $K_3$ is a sufficiently large constant,
we have the inequality
\[-c\partial_{x}\vert{}W_\varepsilon\vert-\varepsilon\partial_{xx}\vert{}W_\varepsilon\vert-2\left(K_4+K_5|x|^{\frac{1}{2}}\right)\vert{}W_\varepsilon\partial_{x}|W_\varepsilon{}|\vert+2\vert{}W_\varepsilon\vert^3\leq{}K_6+K_{7}|W_\varepsilon|,\]
for positive constants $K_i$, $i=4,5,6,7$. This implies, for large enough $\Theta$ (depending on $K_i$) that
\[-c\partial_{x}\vert{}W_\varepsilon\vert-\varepsilon\partial_{xx}\vert{}W_\varepsilon\vert-2\left(K_4+K_{5}|x|^{\frac{1}{2}}\right)\vert{}W_\varepsilon\partial_{x}|W_\varepsilon{}|\vert+2(|W_\varepsilon\vert-\Theta)^3\leq{}0.\]
We let $H(x)=\frac{R^2}{R^2-|x|^2}$ and check that $\overline{W}(x)=\Theta+H(x)$ is a strict super solution of the preceding differential inequality. Since $\partial_{x}\overline{W}(x)=\frac{2x}{R^2}H(x)^2$ and $\partial_{xx}\overline{W}(x)=\frac{2}{R^2}H(x)^2+\frac{8x^2}{R^2}H(x)^3$, see that, for $|x|<R$ we have
\begin{align*}
    &-c\partial_{x}\overline{W}-\varepsilon\partial_{xx}\overline{W}-2(K_4+K_5|x|^{\frac{1}{2}})\vert\overline{W}\partial_{x}\overline{W}\vert+2(\overline{W}(x)-\Theta)^3\\
    &=K_8\left(-\frac{cx}{R^2}H(x)^2-\frac{2\varepsilon}{R^2}H(x)^2-\frac{8x^2\varepsilon}{R^2}H(x)^3\right)-2K_8^2(K_4+K_5|x|^\frac{1}{2})\frac{2|x|}{R^2}(H(x)^2\Theta+H(x)^3)+2K_8^3H(x)^3\\
    &\geq{}\left(-K_8\frac{c}{R}-\frac{2\varepsilon}{R^2}-\frac{8\varepsilon}{R^2}-\frac{4K_8^2}{R}(K_4+K_5R^{\frac{1}{2}})(\Theta+1)+2K_8^3\right)H(x)^3.
\end{align*}
Therefore, for $K_{8,R}$ chosen sufficiently large depending on $K_i$ and $R$, we have that, for $\varepsilon<1$,
\[-c\partial_{x}\overline{W}-\varepsilon\partial_{xx}\overline{W}-2\left(K_4+K_{5}|x|^{\frac{1}{2}}\right)\overline{W}|\partial_{x}\overline{W}|+2(\overline{W}-\Theta)^3>0.\]
Since $|W_\varepsilon|\leq\frac{1}{2\sqrt{K_2}}|\partial_{x}\psi_\varepsilon|$, and the bounds \eqref{eqn:PsiUniBound} imply $\psi_\varepsilon$ has at least one critical point, we have that there exists an $x^{*}$ such that $|W_\varepsilon|(x^{*})=0$. We note that for $R>1$ we have that $K_{8,R}<K_8$ for a constant $K_8$ independent of $R$.

We now suppose that $\omega=|W_\varepsilon|(x)-\overline{W}(x)$ attains a maximum at  a point $z\in{}B_R(x^{*})$, which is necessarily in the interior. If $R$ is small enough, then the maximum $\omega(z)$ is negative, since $W_\varepsilon$ is continuous and $\overline{W}>\Theta+K_{8,R}>0$. We therefore take $R$ large enough that the maximum of $\omega(x)$ in $B_R(x^{*})$ is exactly $0$. If such an $R$ does not exist, we are done. If such an $R$ does exist then, we have that $\partial_{x}\overline{W}(z)=\partial_{x}|W_\varepsilon|(z)$, $\overline{W}(z)=|W_\varepsilon|(z)$ and $-\partial_{xx}(|W|-\overline{W})(z)\geq{0}$.  We deduce, using the above relations and differential inequalities satisfied by $|W_\varepsilon|$ and $\overline{W}$, that 
\[2(|W_\varepsilon|-\Theta)^3-2(\overline{W}-\Theta)^3<{0},\]
which is only possible if $|W_\varepsilon|(z)<\overline{W}(z)$ which contradicts $\overline{W}(z)=|W_\varepsilon|(z)$. Therefore $|W|(x)<\overline{W}(x)$ for all $x\in\mathbb{R}$. Sending $R\rightarrow{}\infty$ yields $|W_\varepsilon|(x)<\Theta+K_8$ for $x\in\mathbb{R}$.

Thus we have a Lipschitz bound. 
Then, by the Arzela-Ascoli theorem, we conclude locally uniform convergence to a continuous function which is the viscosity solution of \eqref{eqn:HJ_Adv}. The constraint follows from the normalization.

\end{proof}

\section*{Acknowlegement} The research of MHD was funded by an EPSRC Research Grant  EP/V038516/1 and a Royal Society International Exchange Grant IES$\backslash$ R3$\backslash$ 223047. A UKRI Future Leaders Fellowship supported F.S., grant no. (MR/T043571/1). This research was funded in whole or in part by UKRI (EP/V038516/1, MR/T043571/1). For the purpose of open access, a CC BY public copyright licence is applied to any AAM arising from this submission.
\printbibliography
\end{document}